\documentclass[10pt,twoside, a4paper, english, reqno]{amsart}
\usepackage[dvips]{epsfig}
\usepackage{amscd}
\usepackage{amssymb}
\usepackage{amsthm}
\usepackage{amsmath}
\usepackage{latexsym}

\usepackage{upref}
\usepackage[backref]{hyperref}

\usepackage{color}
\usepackage[usenames,dvipsnames]{xcolor}
\theoremstyle{plain}
\newtheorem{thm}{Theorem}[section]
\theoremstyle{plain}
\newtheorem{lem}[thm]{Lemma}
\newtheorem{prop}[thm]{Proposition}

\theoremstyle{definition}
\newtheorem{defi}{Definition}[section]
\newtheorem*{rem}{Remark}

\newenvironment{Assumptions}
{%
\setcounter{enumi}{0}

\begin{enumerate}}%
{\end{enumerate} }
{\end{enumerate} }

\newcommand{\eps}{\ensuremath{\varepsilon}}
\newcommand{\R}{\ensuremath{\Bbb{R}}}

\newcommand{\rd}{\ensuremath{\mathbb{R}^d}}
\newcommand{\supp}{\ensuremath{\mathrm{supp}\,}}
\newcommand{\goto}{\ensuremath{\rightarrow}}
\newcommand{\grad}{\ensuremath{\nabla}}

\def\N{{I\!\!N}}

\numberwithin{equation}{section} \allowdisplaybreaks

	\title[weak-in-time formulation]{Stochastic conservation laws: weak-in-time formulation and strong entropy condition}

\date{}

\author[Imran H. Biswas]{Imran H. Biswas}
\address[Imran H. Biswas]{\newline
 Centre for Applicable Mathematics,
 Tata Institute of Fundamental Research,
  P.O.\ Box 6503, GKVK Post Office,
  Bangalore 560065, India}
\email[]{imran@math.tifrbng.res.in}

\author[Ananta K. Majee]{Ananta K. Majee}
\address[Ananta K. Majee]{\newline
 Centre for Applicable Mathematics,
 Tata Institute of Fundamental Research,
  P.O.\ Box 6503, GKVK Post Office,
  Bangalore 560065, India}
\email[]{ananta@math.tifrbng.res.in}


\thanks{The authors would like to profoundly  thank  K.~H.~Karlsen for his help and encouragement.}

\begin{document}

\begin{abstract}
This article is an attempt to complement some recent developments on conservation laws with stochastic forcing. In a pioneering development,  Feng $\&$ Nualart\cite{nualart:2008} have developed the entropy solution theory for such problems and the presence  of stochastic forcing necessitates introduction of {\it strong entropy condition}. However, the authors'  formulation of entropy inequalities are weak-in-space but strong-in-time. In the absence of a-priori path continuity for the solutions, we take a critical outlook towards this formulation and offer an entropy formulation which is weak-in-time and weak-in-space. 
 \end{abstract}

\maketitle

\section{Introduction}
Let $\big(\Omega, P, \mathcal{F}, \{\mathcal{F}_t\}_{t\ge 0} \big)$ be a filtered probability space which satisfies the usual
 hypothesis. We are interested in finding an  $L^2(\R^d)$( or an appropriate function space)-valued predictable process $u(t)$ which satisfies the stochastic partial differential equation 
 \begin{align} du(t,x) + \mbox{div}_x F(u(t,x)) \,dt&= 
 \sigma(x,u(t,x))\,dW(t)\quad t>0, ~ x\in \R^d, \label{eq:brown_stochconservation_laws}
\end{align}
   with the initial condition 
   
   \begin{align}
   \label{initial_cond} u(0,x) = u_0(x), \quad \quad x\in \R^d.
   \end{align} 
    In the above, $W(t)$ is an one-dimensional standard Brownian motion,  $ F:\R\rightarrow \R^d$ is  the flux function,
and $\sigma(x,u)$ is real valued function defined on the domain $\R^d \times \R$.
 
 In the case where $\sigma= 0$, the equation \eqref{eq:brown_stochconservation_laws} becomes a standard scalar
  conservation law with spatial dimension $d$.  It is well-known for conservation laws that solutions (that are obtained by method of characteristics) may develop discontinuities in finite time even when the initial data is smooth.  In other words, the problem \eqref{eq:brown_stochconservation_laws}-\eqref{initial_cond} does not have smooth solutions in general, even when the right hand side is zero.   
  In this situation one has to invoke the notion of weak solutions, but the issues would persist as there could be infinitely many weak solutions to a given problem.  It was a huge step forward for analytical understanding for scalar conservation laws when Kruzkov came up with his idea of entropy solutions.  
   Kruzkov's notion of entropy condition correctly isolates the physically relevant solution in a unique way, and there is a large body of
literature (see \cite{godu,dafermos} and references therein) that has emerged on this subject.

Stochastic conservation laws is a relatively new area of pursuit. It is only recently that conservation laws with stochastic forcing have attracted the attention  
 of many authors (\cite{nualart:2008,KIm2005, risebroholden1997, Vovelle2010, vallet2010, Chen-karlsen2011,xu}), and resulted in significant momentum in the theoretical  development of such problems.  As its deterministic counterpart, it is required to have a weak formulation coupled with an entropy criterion to establish the  wellposedness for such problems.  The equations of type
\eqref{eq:brown_stochconservation_laws} 
     could be interpreted as the equation that describes conservation of physical quantities that are subjected to random force fields
     modeled by diffusion noise.  One of the early work in this direction was \cite{risebroholden1997}, where one dimensional stochastic balance laws
were studied where $\sigma$ is independent of $x$. The authors employed the splitting method to construct approximate solutions, and the approximations were shown to converge to a weak (possibly non-unique) solution. At around the same time, Khanin {et al.} \cite{Sinai1997} published a very influential article  describing some statistical properties of Burger's equations with stochastic forcing. When the noise term on the right hand side is of  additive nature i.e. $\sigma\equiv \sigma(t,x)$, J.~ U.~ Kim\cite{KIm2005} extended Kruzkov's entropy formulation and established the wellposedness for one dimensional problems under the assumption that $\sigma\in C((0,\infty): W_x^{1,\infty})$ and has compact support. The straight forward adaptation of the  deterministic entropy inequalities fails to capture the noise-noise interaction, and the standard mechanism to derive the $L^1$-contraction 
principle does not apply for  general $\sigma$. This issue was finally resolved by Feng $\&$ Nualart\cite{nualart:2008} with the
introduction of the notion of {\em strong entropy } solution.  In \cite{nualart:2008}, the authors established the uniqueness
of strong entropy solution in $L^p$-framework for several space dimension.  The existence, however, was restricted to one space
dimension. We also refer  to the recent articles by Vovelle $\&$ Debussche\cite{Vovelle2010} and by
Chen {\em et al.}\cite{Chen-karlsen2011} for the existence in the multi dimensional case. In \cite{Vovelle2010},
the authors obtain the existence via kinetic formulation. In \cite{Chen-karlsen2011}, the authors use the BV solution framework. In
this paper, we offer a weaker entropy formulation for \eqref{eq:brown_stochconservation_laws} and establish wellposedness
 in the $L^p$-framework. In addition, we refer to \cite{walsh,stroock,Protter1990,parthasarathy,metivier,malek,evanspde} for additional details relevant to the topic.

 In our view, the article \cite{nualart:2008} by Feng $\&$ Nualart is no less than a milestone in the subject and presents a comprehensive theory of entropy solutions for stochastic conservation laws. We draw our primary motivation from \cite{nualart:2008}, but take a critical outlook to the approach and raise a few objections to some of the methods and offer an alternative which we perceive as better suited to the problem.  The ordinary entropy inequalities in the stochastic case do not fully capture the noise-noise interactions and it may not be possible to replicate Kruzkov's approach to get the $L^1$-contraction principle. This issue is resolved by  Feng $\&$ Nualart by introducing an additional condition called  {\it strong entropy condition}. However, the entropy inequalities in \cite{nualart:2008} could be described as weak in space but strong in time. Moreover, the strong entropy condition is related to this formulation and reflects the strong-in-time picture. This formulation easily leads to the 
$L^1$ contraction principle, and uniqueness for such formulation naturally follows. However, the question of existence becomes much more subtle. As its deterministic counterpart, the existence is settled via vanishing viscosity method in 
 \cite{nualart:2008} and this is where our viewpoint deviates from that of  \cite{nualart:2008}. The proof of  existence in \cite{nualart:2008} requires the vanishing viscosity approximates to converge {\em for all} time points and the authors have made attempts to justify the convergence for all time points. However, a careful analysis reveals that convergence is established for {\em almost every} time points. This puts a question mark against the validity of the results in \cite{nualart:2008}. 
 The nature of compactness of  vanishing viscosity approximates finds a perfect match with the entropy formulation which is weak in time and space both, and which would coincide with entropy formulation of  \cite{nualart:2008} if the solution process have continuous sample paths. In  \cite{nualart:2008}, the authors make an attempt to establish path-continuity for the vanishing viscosity limit but there are flaws in the proof.  We have added a separate section in this article where we explain these flaws and describe the implications in details. With this apparent inconsistencies in mind, we find it necessary that entropy inequalities are formulated weak in time and space both, and the strong entropy condition has to be accordingly specified to capture noise-noise interaction. In this article we set out to exactly do that.

         The rest of the paper is organized as follows. In the next section, we describe the technical framework, define the 
notion of strong entropy solution for \eqref{eq:brown_stochconservation_laws}-\eqref{initial_cond} and state the main 
theorems. In Section 3, we establish the uniqueness of strong entropy solution of \eqref{eq:brown_stochconservation_laws}-\eqref{initial_cond} 
by deriving the $L^{1}$ contraction property.  In Section 4, we briefly discuss the wellposedness of vanishing viscosity 
approximation of \eqref{eq:brown_stochconservation_laws} and establish the existence of entropy solution
for \eqref{eq:brown_stochconservation_laws}-\eqref{initial_cond}.  In the Section 5, we show that the vanishing
viscosity solution is indeed a strong entropy solution. Finally, in the last section we describe the issues related to the path continuity of vanishing viscosity limit and its implications. We close this section with a description of the notations and 
symbols and the list of assumptions.

By
$C, K$ etc.,  we mean various constants which may change from line to line.
The Euclidean norm on any $\mathbb{R}^d$-type space is
denoted by $|\cdot|$. Furthermore, let $\Pi_T = (0, T)\times \R^d.$ In the rest of the paper, the following assumptions hold: 

\vspace{.5cm}

\begin{Assumptions}
\item \label{A1} For every $ k = 1,2...,d $, the function $F_k(s) \in C^2(\R) $, and $F_k(s), F_k^\prime(s) $ and $F_k^{\prime\prime}(s)$  have at most polynomial growth in $ s$. 
\item \label{A2} There exists a positive constant $C > 0$ such that 
\begin{align*}|\sigma(y,v)-\sigma(x,u)|\leq C (|u-v|+ |x-y|).
\end{align*} 
\item  \label{A3} There exists  a  nonnegative function $ g\in L^\infty(\R^d)\cap L^2(\R^d)$  such that 
\begin{align*}
    |\sigma(x,u)| \le g(x)(1+|u|).
    \end{align*} 
    \item\label{A4}  The set $\{r\in \R: F^{\prime\prime}(r) \neq 0 \}$ is dense in $\R$.

\end{Assumptions}
\section{Technical framework and statements of the main results}

The notion of entropy solution is built around the so called entropy flux pairs.  We begin this section with the definition of entropy flux pairs.
\begin{defi}[entropy flux pair]
$(\beta,\zeta) $ is called an entropy flux pair if $ \beta \in C^2(\R) $ and $\beta \ge0$, and $\zeta = (\zeta_1,\zeta_2,....\zeta_d):\R \mapsto\rd $ is a vector field satisfying
\[\zeta'(r) = \beta'(r)F'(r). \]
 An entropy flux pair $(\beta,\zeta)$ is called convex if $ \beta^{\prime\prime}(s) \ge 0$.  
\end{defi}

As in the deterministic case, the primary motivation behind the notion of entropy solution comes from parabolic 
regularization. However, it requires  considerable amount of work (cf. \cite{nualart:2008}) to show that perturbation by small diffusion will indeed regularize the solutions.
To proceed, we assume that $ u $  is a smooth predictable solution of the parabolic perturbation 
of \eqref{eq:brown_stochconservation_laws} i.e $u$ satisfies 
\begin{align}
 d u(t,x) + \mbox{div}_x F(u(t,x))\,dt& = \sigma(x,u(t,x))\,dW(t) + \eps\Delta u(t,x)\,dt,
\label{eq:heuristic_parabolic} 
\end{align} where $\eps > 0$ is a small positive number.
As compared to the deterministic case, we need to replace the deterministic chain rule for derivatives by It\^{o} chain rule to derive the entropy inequalities.
Let $(\beta,\zeta ) $ be a convex  entropy flux pair. Then, by It\^{o} formula, we have
\begin{align}
  &d\beta(u(t,x) ) + \mbox{div}_x \zeta(u(t,x))\,dt \notag \\
& = \sigma(x,u(t,x))\beta^{\prime}(u(t,x))\,dW(t)+ \frac{1}{2} \sigma^2(x,u(t,x))\beta^{\prime\prime}(u(t,x))\,dt
\notag \\
& \hspace{2cm}+ \Big(\eps \Delta_{xx}\beta(u(t,x)) -\eps \beta''(u(t,x))|\nabla_x u(t,x)|^2\Big)\,dt .\notag
\end{align}
For  each $0\leq \psi\in C_{c}^{1,2}([0,\infty)\times \rd) $, we apply It\^{o} product rule to obtain 

\begin{align}
  d\big(\beta(u(t,x))\psi(t,x) \big) =& \partial_t\psi(t,x) \beta(u(t,x)) \,dt -\psi(t,x)\mbox{div}_x \zeta(u(t,x))\,dt \notag \\
& +\psi(t,x)\sigma(x,u(t,x))\beta^{\prime}(u(t,x))\,dW(t) 
+ \frac{1}{2}\psi(t,x)\,\sigma^2(x,u(t,x))\beta^{\prime\prime}(u(t,x))\,dt 
\notag \\
& +\psi(t,x) \Big(\eps \Delta_{xx}\beta(u(t,x)) -\eps \beta''(u(t,x))|\nabla_x u(t,x)|^2\Big)\,dt .\notag
\end{align}
It is to be kept in mind that  $\beta$ is non-negative and  convex and $\psi$ is non-negative.  Therefore, for every $T> 0$, we have 
\begin{align}
 \notag 0& \le \langle \beta(u(T,.)),\psi(T,\cdot)\rangle \\  \leq  &\langle \beta(u(0,.)),\psi(0,\cdot)\rangle+ 
 \int_{0}^{T} \langle \zeta(u(r,.)),\nabla_x \psi(r,\cdot)\rangle\,dr
 \notag \\
 &+\int_{(0,T]} \langle \beta(u(r,\cdot)),\partial_t \psi(r,\cdot)\rangle\,dr  
  +\int_{(0,T]} \langle \sigma(\cdot,u(r,\cdot))\beta^{\prime}(u(r,\cdot)),\psi(r,\cdot)\rangle \,dW(r)\notag \\
 &+\frac{1}{2}\int_{(0,T]} \langle \sigma^2(\cdot,u(r,\cdot))\beta^{\prime\prime}(u(r,\cdot)),\psi(r,\cdot)\rangle \,dr+  \mathcal{O}(\eps).\label{eq:entropy_derivation}
\end{align}
Both the left-hand and right-hand sides of the inequality are stable under $\eps \goto 0 $, provided we have $ L_{\mbox{loc}}^p $ type stability of \eqref{eq:heuristic_parabolic} as $\eps \goto 0$. The above inequality leads to the entropy inequalities which are weak in time and space both. 

\begin{defi} [stochastic entropy solution]\label{defi:stochentropsol}
An $ L^2(\rd)$-valued $\{\mathcal{F}_t: t\geq 0 \}$-predictable stochastic process $u(t)= u(t,x)$ is
called a stochastic entropy solution of \eqref{eq:brown_stochconservation_laws} provided \\
(1) for each $ T>0$, $ p=2,3,4......$ \[\sup_{0\leq t\leq T} E[||u(t)||_{p}^{p}] <
\infty. \] 

(2) For  $0\leq \psi\in C_{c}^{1,2}([0,\infty )\times \R^d) $ and each convex entropy pair 
$(\beta,\zeta) $, 
\begin{align}
 &\int_{\R^d} \beta(u_0 (x))\psi(0,x)\,dx+  \int_{\Pi_T} \beta(u(t,x))\partial_t \psi(t,x)\,dt\,dx\notag \\
  & +  \int_{\Pi_T}\zeta( u(t,x)) \cdot \grad \psi(t,x)\,dt\,dx
   + \frac{1}{2} \int_{\Pi_T} \sigma^2(x,u(t,x)) \beta^{\prime \prime}(u(t,x))\psi(t,x)\,dx\,dt\notag \\
 &+  \int_0^T \int_{\R^d} \sigma(x,u(t,x)) \beta^\prime(u(t,x))\psi(t,x)\,dx\,dW(t)
  \ge 0\quad P- \text{a.s.}\notag 
\end{align}
\end{defi} 

In the deterministic case, the entropy inequalities  lead  to
the $L^1$-contraction principle which implies uniqueness. In the stochastic case, however, the entropy inequalities alone do not
seem to give rise to desired $L^1$-contraction principle. The Definition \ref{defi:stochentropsol}  does 
not reveal much about the noise-noise interaction  when one tries to compare two solutions of the same problem.
We refer to \cite{Chen-karlsen2011} for detailed mathematical description of this issue. However, to ensure uniqueness, we
need to arrive  at a version of so-called {\em strong entropy condition} which is compatible with the weak-in-time formulation.   

  Let $\rho$ and $\varrho$ be the standard mollifiers on $\R$ and  $\R^d$ respectively such that
  $\supp(\rho) \subset [-1,0]$ and $\supp(\varrho) = B_1(0)$. For $\delta > 0$ and $\delta_0 > 0$, let $\rho_{\delta_0}(r) = \frac{1}{\delta_0}\rho(\frac{r}{\delta_0})$ and  $\varrho_{\delta}(x) = \frac{1}{\delta^d}\varrho(\frac{x}{\delta})$.
For a nonnegative test function $\psi\in C_c^{1,2}([0,\infty)\times \rd)$ and two positive constants $\delta, \delta_0 $, define 
             
 \begin{align}
\label{eq:doubled-variable} \phi_{\delta,\delta_0}(t,x; s,y) = \rho_{\delta_0}(t-s) \varrho_{\delta}(x-y) \psi(s,y). 
\end{align} 
  Note that $ \rho_{\delta_0}(t-s) \neq 0$ only if $s-\delta_0 \le t\le s$, and therefore $ \phi_{\delta,\delta_0}(t,x; s,y)= 0$ outside  $s-\delta_0 \le t\le s$.

\begin{defi} [stochastic strong entropy solution]\label{strong-entropy-condition}
 An $ L^2(\rd)$-valued $\{\mathcal{F}_t: t\geq 0 \}$-predictable stochastic process $ v(t)=v(t,x)$ is
called a stochastic strong entropy solution of \eqref{eq:brown_stochconservation_laws} provided 
\begin{itemize}
\item[$(i)$] it is a stochastic entropy solution. 
\item[$(ii)$] For each $ L^2(\rd)$-valued $\{\mathcal{F}_t: t\geq 0 \}$-adapted stochastic process $
\tilde{u}(t,x)$ satisfying, for  $ T>0$, $ p=2,3,4......$ \[
\sup_{0\leq t\leq T}E\Big[||\tilde{u}(t)||_{p}^{p}\Big ] < \infty,\]
and for each $\beta \in C^{\infty}(\R) $ such that $\beta^{\prime\prime}$ and $ \beta^{\prime\prime\prime}$ are of compact support and  $ 0\leq \psi\in C_{c}^{\infty}([0,\infty)\times\rd) $,
and \[h(r,s;v,y)=\int_{x}\sigma(x,\tilde{u}(r,x))\beta^{\prime}(\tilde{u}(r,x)-v)\phi_{\delta,\delta_0}(r,x;s,y)\,dx,\]
where $\phi_{\delta,\delta_0} $ is defined by in \eqref{eq:doubled-variable},

\begin{align}
&E\Big[\int_0^T\int_{y}\Big[\int_0^T h(r,s;v,y)\,dW(r)\Big]_{v=v(s,y)}\,dy\,ds\Big] \notag \\
&\leq - E\Big[\int_{\Pi_T}\int_{\Pi_T}  \sigma(x,\tilde{u}(r,x))\sigma(y,v(r,y))\beta^{\prime\prime}(\tilde{u}(r,x)-v(r,y))\notag \\
&\hspace{3cm}\times \phi_{\delta,\delta_0}(r,x,s,y)
 \,dr\,dx \,dy \,ds \Big] +  A(\delta,\delta_0), \notag
\end{align}
where $A(\delta,\delta_0)$ is a function depending on $\beta,~\psi$ such that  $A(\delta,\delta_0)\goto 0 $ as $\delta_0 \goto 0 $.
\end{itemize}

\end{defi} 
\begin{rem}
  The weak-in-time formulation is also manifested in the {\it strong entropy condition}. In our formulation the function $A(\delta,\delta_0)$ plays a similar role as that of $A(s,t)$ in Feng $\&$ Nualart. 
\end{rem}
The above definition does not say anything explicitly about the entropy solution satisfying the initial condition. However, it 
satisfies the initial condition in a certain weak sense. We have the following lemma.
\begin{lem}\label{lem:initial-cond}
 Any entropy solution $u(t,x)$ of \eqref{eq:brown_stochconservation_laws} satisfies the initial condition in the following sense:
  for every non negative $\psi\in C_c^2(\R^d)$ such that $\supp(\psi) = K$,
  \begin{align*}
    \lim_{h\rightarrow 0}E \Big[\frac 1h \int_0^h\int_K| u(t,x) -u_0(x)| \psi(x)\,dx\, dt \Big]= 0. 
  \end{align*}
\end{lem}

\begin{proof}
    Since $K$ is of finite measure, it is enough if we instead prove 
   \begin{align}
  \lim_{h\rightarrow 0}E \Big[\frac 1h \int_0^h\int_K |u(t,x) -u_0(x)|^2 \psi(x)\,dx\, dt \Big]= 0.  \label{eq:intial_cond_weak-1}
\end{align}

For $\delta\in (0,1)$, let $K_\delta = \{x: \text{dist}(x,K) \le \delta \}$.  Note that, for any $\delta > 0$,

\begin{align}
 E\int_{K}  | u(t,x) -u_0(x)|^2\psi(x)\,dx 
     &\le 2 E\int_{y\in K_\delta}\int_{x\in K} |u(t,x) -u_0(y)|^2\psi(x)  \varrho_{\delta}(x-y)\,dx\,dy\notag\\
     &\quad+2 E\int_{y\in K_\delta}\int_{x\in K} |u_0(y) -u_0(x)|^2 \psi(x) \varrho_{\delta}(x-y)\,dx\,dy.\label{eq:intial_cond_weak-3}
 \end{align} where $\{\varrho_\delta\}$ is a sequence of mollifiers in $\R^d$. In other words
   \begin{align}
& E\frac{1}{h}\int_0^h\int_{K} |u(t,x) -u_0(x)|^2\psi(x)\,dx \,dt\notag\\
     \le& 2 E\frac{1}{h}\int_0^h\int_{y\in K_\delta}\int_{x\in K}|u(t,x) -u_0(y)|^2\psi(x) \varrho_{\delta}(x-y)\,dx\,dy\,dt\notag\\
     &\quad+2 E\int_{y\in K_\delta}\int_{x\in K} |u_0(y) -u_0(x)|^2 \psi(x) \varrho_{\delta}(x-y)\,dx\,dy.\label{eq:intial_cond_weak-4}
 \end{align} Now let $\psi(t,x)= \gamma(t)\psi(x)\varrho_\delta(x-y)$, where $\gamma(t)= \frac{h-t}{h}$ for $0\le t\le h$. 
 Now, let $\beta(u) = (u -u_0(y))^2$ and $\xi(u)= \int_0^u 2 (r-u_0(y)) F^\prime(r)\, dr = 2\int_0^u r F^\prime(r)\, dr -2u_0(y)(F(u)-F(0))\le C (1+|u_0(y)|^2+ |u|^p) $ 
 for some positive integer $p$. With the  above entropy flux pair $(\beta,\xi)$, we apply Definition  \ref{defi:stochentropsol} and have

\begin{align*}
&E\frac{1}{h}\int_0^h\int_{y\in K_\delta}\int_{x\in K} |u(t,x) -u_0(y)|^2\psi(x)  \varrho_{\delta}(x-y)\,dx\,dy\,dt\\
\le&  E\int_{y\in K_\delta}\int_{x\in K} |u_0(y) -u_0(x)|^2 \psi(x) \varrho_{\delta}(x-y)\,dx\,dy \\
&\quad+ C\delta^{-2}\int_0^hE \int_{y\in K_\delta}\int_{x\in K}(1+|u(r,x)|^p+|u_0(y)|^2) \,dx\,dy\,dr \\
  &+\frac{C^{\prime\prime}}{\delta}\int_0^hE \int_{x\in K} \sigma^2(x,u(r,x)) \,dx\,dr.
\end{align*} Hence by passing to the limit $h\rightarrow 0$, we have

\begin{align}
&\limsup_{h\rightarrow 0} E\frac{1}{h}\int_0^h\int_{y\in K_\delta}\int_{x\in K} |u(t,x) -u_0(y)|^2\psi(x)  \varrho_{\delta}(x-y)\,dx\,dy\,dt\notag\\
\le &  E\int_{y\in K_\delta}\int_{x\in K} |u_0(y) -u_0(x)|^2 \psi(x) \varrho_{\delta}(x-y)\,dx\,dy.\label{eq:intial_cond_weak-5}
\end{align} We combine \eqref{eq:intial_cond_weak-4} and \eqref{eq:intial_cond_weak-5} and obtain

\begin{align}
& \limsup_{h\rightarrow 0} E\frac{1}{h}\int_0^h\int_{K} |u(t,x) -u_0(x)|^2\psi(x)\,dx \,dt \notag\\
\le &4 E\int_{y\in K_\delta}\int_{x\in K} |u_0(y) -u_0(x)|^2 \psi(x) \varrho_{\delta}(x-y)\,dx\,dy\quad \text{for all} ~\delta > 0\label{eq:intial_cond_weak-6}.
\end{align}We now simply let $\delta \rightarrow 0$ in the RHS of \eqref{eq:intial_cond_weak-6} and obtain 
\begin{align*}
\limsup_{h\rightarrow 0} E\frac{1}{h}\int_0^h\int_{K} |u(t,x) -u_0(x)|^2\psi(x)\,dx \,dt\le 0.
\end{align*} Hence \eqref{eq:intial_cond_weak-1} follows as $\psi \ge 0$.
 This completes the proof.
\end{proof}

 Next, we describe a special class of entropy functions that plays an important role in the sequel.  Let $\beta:\R \rightarrow \R$ be a nonnegative smooth function  satisfying 
 \begin{align*}
      \beta(0) = 0,\quad \beta(-r)= \beta(r),\quad \beta^\prime(-r) = -\beta^\prime(-r),\quad \beta^{\prime\prime} \ge 0,
 \end{align*} and 
\begin{align*}
\beta^\prime(r)=\begin{cases} -1\quad \text{when} ~ r\le -1,\\
                             \in [-1,1] \quad\text{when}~ |r|<1,\\
                             +1 \quad \text{when} ~ r\ge 1.
                \end{cases}
\end{align*} For any $\epsilon > 0$, define  $\beta_\epsilon:\R \rightarrow \R$ by 
\begin{align*}
         \beta_\epsilon(r) = \epsilon\beta(\frac{r}{\epsilon}).
\end{align*} Then
\begin{align}\label{eq:approx to abosx}
 |r|-M_1\epsilon \le \beta_\epsilon(r) \le |r|\quad \text{and} \quad |\beta_\eps^{\prime\prime}(r)| \le \frac{M_2}{\epsilon} {\bf 1}_{|r|\le \epsilon},
\end{align} where
\begin{align*}
 M_1 = \sup_{|r|\le 1}\big | |r|-\beta(r)\big |, \quad M_2 = \sup_{|r|\le 1}|\beta^{\prime\prime} (r)|.
\end{align*}

By simply dropping $\epsilon$, for $\beta= \beta_\eps$ ~ we define 
\begin{align*}
&F_k^\beta(a,b)=\int_{b}^a \beta^\prime(\sigma-b)F_k^\prime(\sigma)\,d(\sigma), \\
&F^\beta(a,b)=(F_1^\beta(a,b),F_2^\beta(a,b),...,F_d^\beta(a,b)),\\
&F_k(a,b)= \text{sign}(a-b)(F_k(a)-F_k(b)) ,\\
&F(a,b)= (F_1(a,b),F_2(a,b),....,F_d(a,b)).
\end{align*} 


 We are now ready to state the main results of this paper. 
 
\begin{thm}[uniqueness] \label{thm:uniqueness}
 Let the assumptions \ref{A1}-\ref{A3} be true, and that $\cap_{p=1,2,..} L^p(\rd)$-valued and   $\mathcal{F}_0$-measurable random variable  $ u_0$ satisfies
\[E \Big[||u_0||_{p}^{p} + ||u_0||_{2}^{p}\Big] < \infty  \qquad\text{for}~ p=1,2,...~ .\]
Suppose that $ u$, $ v$ be two stochastic entropy solutions of \eqref{eq:brown_stochconservation_laws} with the same initial condition $ u(0)=u_0 =v(0)$, and that at least one of $ u $,
$ v$ is a strong stochastic entropy solution.  Then almost surely $ u(t)=v(t) $ for almost every $t \ge  0$.  
\end{thm}
 We further assume that $d =1$, and  state the existence theorem of strong entropy solution.
\begin{thm}[existence] \label{thm:existence}
 Let the assumptions \ref{A1}-\ref{A4} be true and  $d =1$. Furthermore, $\cap_{p=1,2,..} L^p(\rd)$-valued $\mathcal{F}_0$-measurable random variable  $ u_0$ satisfies
\[E \Big[||u_0||_{p}^{p} + ||u_0||_{2}^{p}\Big] < \infty  \qquad\text{for}~ p=1,2,...~ .\]
Then there exists a strong entropy solution for \eqref{eq:brown_stochconservation_laws} -\eqref{initial_cond}
\end{thm}

\begin{rem}
In the sequel it is going to be clear that our results are still valid if the noise is of the form $\sum_{i=1}^{m} \sigma_i(x, u) \,d W_i(t)$. This is a special case of space-time noise $\int_{z}\sigma(u, x, z)\partial_tW(t, \,dz)$ in \cite{nualart:2008}. This space-time noise structure does have close resemblance with L\'{e}vy/ pure jump type noise structure $\int_{z}\tilde{N}(\, dz, \, dt)$. From our recent experience of working with  conservation laws with L\'{e}vy noise, we confidently infer that our results could be extended to the generalized noise structure of \cite{nualart:2008}. 
\end{rem}

\section{Proof of uniqueness}
  The proof of uniqueness follows a line argument that suitably adapts Kruzkov's method of doubling the variables to the stochastic  case. The central idea of the proof is to analyze the evolution of $||u(t)-v(t) ||_{L^1(\R^d)}$ as a random quantity, and then arrive at the conclusion that $E\big(||u(t)-v(t) ||_{L^1(\R^d)}\big)$ decreases as a function of time. In our context also we use doubling of variables, and approximate $||u(s)-v(s) ||_{L^1(\R^d)}$  by $\int_0^T \int_{\R^d\times \R^d} \beta(u(t,x)-v(t,y))\psi(t) \varphi(x,y)\,dx\,dy\,dt$, where $\beta(r)$ is a suitable smooth convex approximation of $|r|$ and $\varphi(x,y)$ is a smooth approximation for $\delta_x(y)$ and $\psi(t)$ is a smooth approximation of $\delta_s(t)$. We will, however, have to handle
additional difficulties due to the stochastic forcing.

Let $ u$ be a stochastic entropy solution and $v$ be a stochastic strong entropy solution to
equation \eqref{eq:brown_stochconservation_laws}. Let $ 0\leq \psi\in C_{c}^{1,2}([0,\infty)\times\rd)$ be given
and $\beta \equiv \beta_\epsilon$ (as described above). For a fixed real number $k\in \R$, $\beta(\cdot-k)$ is
a convex smooth function. Therefore $(\beta(\cdot-k), F^\beta(\cdot, k))$ could be chosen as the corresponding convex
entropy flux pair where $F^{\beta}(a,b)$ is described above. Next, we lay down the entropy inequality for $u(t,x)$ relative to the convex entropy pair $(\beta(\cdot-k), F^\beta(\cdot, k))$ and substitute $k$ by $v(s,y)$ and integrate with respect to $s, y$ to get

\begin{align}
 &\int_{\Pi_T}\int_{\R^d} \beta(u_0(x)-v(s,y))\phi_{\delta,\delta_0}(0,x,s,y)\,dx\,dy\,ds 
 +\int_{\Pi_T} \int_{\Pi_T}\beta(u(t,x)-v(s,y))\partial_t \phi_{\delta,\delta_0}\,dx\,dt\,dy\,ds\notag \\ 
& +\int_0^T\int_{y}\Big[\int_0^T h(r,s,;v,y) dW(r)\Big]_{v=v(s,y)}\,dy\,ds\notag\\
& +\frac{1}{2} \int_{\Pi_T} \int_{0}^T\int_{\R^d}  \sigma^2(x, u(t,x))\beta^{\prime\prime}(u(t,x)-v(s,y)) 
\phi_{\delta,\delta_0}(t,x;s,y)\,dx\,dt\,dy\,ds\notag \\
& +\int_{\Pi_T}\int_{\Pi_T} F^\beta(u(t,x),v(s,y))\grad_x\phi_{\delta,\delta_0}\,dx\,dt\,dy\,ds \geq 0 \label{stochas_entropy_1},
\end{align} where $\Pi_T = [0,T]\times \R^d$ and 
\[h(r,s;v,y)=\int_{x}\sigma(x, u(r,x))\beta^{\prime}(u(r,x)-v)\phi_{\delta,\delta_0}(r,x;s,y)\,dx.\]

Similarly, since $v(s,y)$ is also a stochastic  entropy solution, by substituting $k=u(t,x)$ and integrating with respect to $(t,x)$ we have
\begin{align}
 &\int_{\Pi_T}\int_{\R^d} \beta(v_0(y)-u(t,x))\phi_{\delta,\delta_0}(t,x,0,y)\,dx\,dy\,dt
 +\int_{\Pi_T} \int_{\Pi_T}\beta(v(s,y)-u(t,x))\partial_s \phi_{\delta,\delta_0}\,dy\,ds\,dx\,dt\notag \\ 
& +\int_{\Pi_T} \int_{0}^T\int_{\R^d}\sigma(y,v(s,y))\beta^\prime(v(s,y)-u(t,x))\phi_{\delta,\delta_0}\,dy\,dW(s)\,dx\,dt\notag\\
& +\frac{1}{2} \int_{\Pi_T} \int_{0}^T\int_{\R^d}  \sigma^2(y, v(s,y))\beta^{\prime\prime}(v(s,y)-u(t,x))
\phi_{\delta,\delta_0}(t,x;s,y)\,dy\,ds\,dx\,dt\notag \\
& +\int_{\Pi_T}\int_{\Pi_T} F^\beta(v(s,y),u(t,x))\grad_y\phi_{\delta,\delta_0}\,dx\,dt\,dy\,ds \geq 0 \label{stochas_entropy_2}
\end{align}
Adding the two inequalities \eqref{stochas_entropy_1} and \eqref{stochas_entropy_2} and using the fact that
$\text{supp}~\rho_{\delta_0}\subset
[-\delta_0,0]$, we notice that the terms involving $\partial_s \rho_{\delta_0}$ and $\partial_t \rho_{\delta_0}$ cancel each other and we are left with
\begin{align}
 &\int_{\Pi_T}\int_{\R^d} \beta(u_0(x)-v(s,y))\psi(s,y)\,\rho_{\delta_0}(-s)\varrho_{\delta}(x-y)\,dx\,dy\,ds \notag \\
 & + \int_{\Pi_T}\int_{\Pi_T}\beta(v(s,y)-u(t,x))\partial_s\psi(s,y)\,\rho_{\delta_0}(t-s)\varrho_{\delta}(x-y) \,dy\,ds\,dx\,dt\notag \\ 
 & +\int_{\Pi_T}\int_{\Pi_T} F^\beta(v(s,y),u(t,x))\grad_y\psi(s,y)\,\rho_{\delta_0}(t-s)\varrho_{\delta}(x-y)\,dx\,dt\,dy\,ds \notag \\
 & +\int_{\Pi_T}\int_{\Pi_T} F^\beta(u(t,x),v(s,y))\grad_x\varrho_{\delta}(x-y)\psi(s,y)\rho_{\delta_0}(t-s)\,dy\,ds\,dx\,dt\notag \\
&+\int_{\Pi_T}\int_{\Pi_T} F^\beta(v(s,y),u(t,x))\grad_y\varrho_{\delta}(x-y)\psi(s,y)\rho_{\delta_0}(t-s)\,dx\,dt\,dy\,ds \notag \\
&+\int_0^T\int_{y}\Big[\int_0^T h(r,s,;v,y) dW(r)\Big]_{v=v(s,y)}\,dy\,ds\notag\\
& +\int_{\Pi_T} \int_t^{t+\delta_0}\int_{\R^d}\sigma(y,v(s,y))\beta^\prime(v(s,y)-u(t,x))\phi_{\delta,\delta_0}(t,x;s,y)
\,dy\,dW(s)\,dx\,dt\notag\\
& +\frac{1}{2} \int_{\Pi_T} \int_{0}^T\int_{\R^d}  \sigma^2(y, v(s,y))\beta^{\prime\prime}(v(s,y)-u(t,x))
\phi_{\delta,\delta_0}(t,x;s,y)\,dy\,ds\,dx\,dt\notag \\
& +\frac{1}{2} \int_{\Pi_T} \int_{0}^T\int_{\R^d}  \sigma^2(x, u(t,x))\beta^{\prime\prime}(u(t,x)-v(s,y))
\phi_{\delta,\delta_0}(t,x;s,y)\,dx\,dt\,dy\,ds\notag \\
 &\geq 0.
\end{align}
We now take expectation on both sides and use the property of  $v(s,y)$ as a  strong entropy solution to have
\begin{align}
 &E\Big[\int_{\Pi_T}\int_{\R^d} \beta(u_0(x)-v(s,y))\psi(s,y)\,\rho_{\delta_0}(-s)\varrho_{\delta}(x-y)\,dx\,dy\,ds \Big]\notag \\
 & +E \Big[\int_{\Pi_T} \int_{\Pi_T}\beta(v(s,y)-u(t,x))\partial_s\psi(s,y)\,\rho_{\delta_0}(t-s)\varrho_{\delta}(x-y) \,dy\,dx\,dt\,ds\Big]\notag \\ 
 & +E\Big[\int_{\Pi_T}\int_{\Pi_T} F^\beta(v(s,y),u(t,x))\grad_y\psi(s,y)\,\rho_{\delta_0}(t-s)\varrho_{\delta}(x-y)\,dx\,dt\,dy\,ds\Big] \notag \\
 & +E\Big[\int_{\Pi_T}\int_{\Pi_T} F^\beta(u(t,x),v(s,y))\grad_x\varrho_{\delta}(x-y)\psi(s,y)\rho_{\delta_0}(t-s)\,dx\,dy\,dt\,ds\Big]\notag \\
&+E\Big[\int_{\Pi_T}\int_{\Pi_T} F^\beta(v(s,y),u(t,x))\grad_y\varrho_{\delta}(x-y)\psi(s,y)\rho_{\delta_0}(t-s)\,dx\,dt\,dy\,ds\Big] \notag \\
& +\frac{1}{2} E\Big[ \int_{\Pi_T} \int_{0}^T\int_{\R^d}\sigma^2(x, u(t,x))\beta^{\prime\prime}\Big(u(t,x)-v(s,y)\Big)
\phi_{\delta,\delta_0}(t,x;s,y)\,dx\,dt\,dy\,ds\Big]\notag \\
& +\frac{1}{2}E\Big[\int_{\Pi_T} \int_{0}^T\int_{\R^d} \sigma^2(y, v(s,y))\beta^{\prime\prime}\Big(v(s,y)-u(t,x)\Big)
 \phi_{\delta,\delta_0}(t,x;s,y)\,dy\,ds\,dx\,dt\Big]\notag \\ 
& -E\Big[\int_{\Pi_T}\int_{\Pi_T} \sigma(x,u(t,x))\sigma(y,v(t,y))\beta^{\prime\prime}(u(t,x)-v(t,y))\psi(s,y)\notag \\
& \hspace{5cm} \times \rho_{\delta_0}(t-s)\varrho_{\delta}(x-y) \,dy\,dx\,dt\,ds\Big] + A(\delta,\delta_0)\notag \\
&\equiv I_1 +I_2 +I_3 +I_4+I_5 +I_6 +I_7 +I_8 +  A(\delta,\delta_0)\notag \\
 & \geq 0 \label{stochas_entropy_3}
\end{align}

Now, we estimate each of the terms above as $\delta_0,\delta \rightarrow 0$ and $\beta \rightarrow |\cdot|$. We start with $I_1$.

\begin{lem}\label{stochastic_lemma_1}
 \begin{align}
 \lim_{\delta_0 \goto 0} I_1 = E\int_{\R^d}\int_{\R^d} \beta(u_0(x)-v_0(y))\psi(0,y)\varrho_{\delta} (x-y)\,dx\,dy\notag
 \end{align} and
  \begin{align*}
  \lim_{(\eps,\delta)\rightarrow (0,0)} E \int_{\R^d\times\R^d}\beta_{\eps} \big(u_0(x)-v_0(y)\big) \varrho_{\delta}(x-y) \psi(0,y) \,dx\,dy= E\int_{\R^d} |u_0(x)-v_0(x)|\,\psi(0,x)\,dx. 
 \end{align*}
\end{lem}
\begin{proof}
 The proof is divided into two steps, and in each step,  we will justify the passage to the corresponding limit.\\
\noindent{\textbf{Step 1:}}  In this step  we consider the passage to the limit as $\delta_0 \goto 0$. Let
\begin{align}
 \mathcal{A}_1:=& E\int_{\Pi_T}\int_{\R^d} \beta(u_0(x)-v(s,y))\psi(s,y)\varrho_{\delta}(x-y)\rho_{\delta_0}(-s)\,dx \,dy\,ds\notag\\
&-  E\int_{\R^d}\int_{\R^d} \beta(u_0(x)-v_0(y))\psi(0,y)\varrho_{\delta}(x-y)\,dx \,dy\notag \\
& = E\int_{\Pi_T}\int_{\R^d} \beta(u_0(x)-v(s,y))[\psi(s,y)-\psi(0,y)]\rho_{\delta_0}(-s)\varrho_{\delta}(x-y)\,dx \,dy\,ds\notag\\
& + E\int_{\Pi_T}\int_{\R^d} \Big[\beta(u_0(x)-v(s,y))-\beta(u_0(x)-v_0(y))\Big]\notag \\
&\hspace{4cm}\psi(0,y)\varrho_{\delta}(x-y)\rho_{\delta_0}(-s)\,dx\,dy\,ds.\notag 
\end{align}
Since support $\psi(s,\cdot)\subset K$, we have
\begin{align}
 | \mathcal{A}_1|\le& ||\psi_t||_{\infty}  E\int_{\Pi_T}\int_{\R^d}\chi_{K}(y) \beta(u_0(x)-v(s,y))\,s\,\rho_{\delta_0}(-s)\varrho_{\delta}(x-y)\,dx \,dy\,ds\notag\\
& + ||\beta^{\prime}||_{\infty}  E\int_{\Pi_T}\int_{\R^d} |v(s,y)-v_0(y)|\psi(0,y)\varrho_{\delta}(x-y)\rho_{\delta_0}(-s) \,dx\,dy\,ds\notag \\
\le&||\psi_t||_{\infty}\,||\beta^{\prime}||_{\infty}\delta_0 E\int_{\Pi_T}\int_{\R^d}\chi_{K}(y)(|u_0(x)-v(s,y)|)\rho_{\delta_0}(-s)\varrho_{\delta}(x-y)\,dx \,dy\,ds\notag\\
& + ||\beta^{\prime}||_{\infty}  E\int_{\Pi_T}\int_{\R^d} |v(s,y)-v_0(y)|\psi(0,y)\varrho_{\delta}(x-y)\rho_{\delta_0}(-s)\,dx\,dy\,ds\notag \\
\le&||\psi_t||_{\infty}\,||\beta^{\prime}||_{\infty}\delta_0~  E\int_{\Pi_T}\int_{\R^d}\chi_{K}(y)(|u_0(x)-v(s,y)| )\,\rho_{\delta_0}(-s)\varrho_{\delta}(x-y)\,dx\,dy\,ds\notag\\
& +||\beta^{\prime}||_{\infty}  E\int_{0}^T\int_{K} \psi(0,y) |v(s,y)-v_0(y)|\,\rho_{\delta_0}(-s)\,dy\,ds\notag \\
\le&||\psi_t||_{\infty}\,||\beta^{\prime}||_{\infty}\delta_0~ \Big[||u_0(x)||_{L^1(\R^d)}+  E\int_{0}^T\int_{K}|v(s,y)|\,\rho_{\delta_0}(-s)\,dy\,ds \Big]\notag\\
& +||\beta^{\prime}||_{\infty}  E\int_{0}^T\int_{K} \psi(0,y)|v(s,y)-v_0(y)|\,\rho_{\delta_0}(-s)\,dy\,ds\notag \\
\le&||\psi_t||_{\infty}\,||\beta^{\prime}||_{\infty}\delta_0~ \Big[||u_0(x)||_{L^1(\R^d)}+\sup_{0\le s\le T}E \big(||v(s,\cdot)||_{L_1}\big) \Big]\notag\\
& + C ||\beta^{\prime}||_{\infty}  \frac{1}{\delta_0}\int_{0}^{\delta_0}E\Big(\int_{K} \psi(0,y)|v( r,y)-v_0(y)|\,dy\Big)\, dr.\notag
\end{align} By Lemma \ref{lem:initial-cond},  $\lim_{\delta_0 \rightarrow 0}\frac{1}{\delta_0}\int_{0}^{\delta_0}E\Big(\int_{K} \psi(0,y)|v( r,y)-v_0(y)|\,dy\Big)\, dr =0$. Therefore, $\lim_{\delta_0\rightarrow 0} \mathcal{A}_1 = 0$.

 \vspace{.1cm}
\noindent{\textbf{Step 2:}}  In this step,
 we now establish the second half of the lemma. Note that the sequence $(\beta_{\eps})_\eps$ is a sequence of functions
 that satisfies $ \Big|\beta_\eps(r)-|r|\Big|\le  C \eps $ for any $r\in \R $.  Therefore
 
 \begin{align*}
 & \Big|E \int_{\R^d\times\R^d}\beta_{\eps} \big(u_0(x)-v_0(y)\big) \varrho_{\delta}(x-y) \psi(0,y) \,dx\,dy-
 E\int_{\R^d} |u_0(y)-v_0(y)|\,\psi(0,y)\,dy \Big|\\
 \le & E \int_{\R^d\times\R^d}\big| \beta_{\eps} \big(u_0(x)-v_0(y)\big)-| u_0(x)-v_0(y)|\big|  \varrho_{\delta}(x-y) \psi(0,y) \,dx\,dy\\
 & +   E \int_{\R^d\times\R^d}\big||u_0(x)-v_0(y)|-| u_0(y)-v_0(y)|\big|  \varrho_{\delta}(x-y) \psi(0,y) \,dx\,dy\\
 \le & \text{Const}(\psi) \eps +  E \int_{\R^d\times\R^d}\big|u_0(x)- u_0(y)| \varrho_{\delta}(x-y) \psi(0,y) \,dx\,dy\\
 \le & \text{Const}(\psi) \eps +  ||\psi||_{\infty}E\int_{|z|\le 1} \int_{\R^d}\big|u_0(x)- u_0(x+\delta z)| \varrho(z) \,dx\,dz.
 \end{align*} Note that $ \lim_{\delta \downarrow 0} \int_{\rd} |u_0(x)-u_0(x+\delta z)| \,dx\rightarrow 0 $ for all  $||z||\le 1$, 
therefore by bounded convergence theorem  we have  $\lim_{\delta \downarrow 0}  E\int_{|z|\le 1}\int_{\rd}
|u_0(x)-u_0(x+\delta z)| \varrho(z) \,dx\,dz = 0 .$ This allows us to pass to the limit $(\eps,\delta)\rightarrow (0,0)$ in the last line and establish the second part of the claim.
\end{proof}

\begin{lem}\label{stochastic_lemma_2} It follows that

 \begin{align} \lim_{\delta_0\goto 0}I_2=\,E\int_{\Pi_T}\int_{\R^d}\beta(v(s,y)-u(s,x))\partial_s\psi(s,y)\varrho_{\delta}(x-y)
 \,dy\,dx\,ds\notag
 \end{align} and 
 \begin{align*}
  &\lim_{(\eps,\delta)\rightarrow (0,0)} E\Big[\int_{\Pi_T}\int_{x} \beta_{\eps}(v(s,y)-u(s,x))
  \partial_s \psi(s,y)\varrho_\delta(x-y)\,dx\,dy\,ds\Big] \\
  &\hspace{1cm}=E\Big[\int_{\Pi_T} \big|v(s,x)-u(s,x)\big|
  \partial_s \psi(s,x)\,dx\,ds\Big].
  \end{align*}
\end{lem}
\begin{proof} As before, the proof is divided into two steps and in each of these steps we will justify the corresponding  passage to the limit.\\
\noindent{\textbf{Step 1:}} Firstly,  we consider the passage to the limit as $\delta_0 \goto 0$. For this, let
\begin{align}
 &\mathcal{G}_1:=\Big|E\int_{\Pi_T}\int_{\Pi_T}\beta(v(s,y)-u(t,x))\partial_s\psi(s,y)\rho_{\delta_0}(t-s)\varrho_{\delta}(x-y)\,dy\,ds\,dx\,dt\notag\\
& \hspace{1cm}-E\int_{\Pi_T}\int_{\R^d}\beta(v(s,y)-u(s,x))\partial_s\psi(s,y)\varrho_{\delta}(x-y)\,dy\,dx\,ds\Big|\notag\\
& =\Big|E\int_{s=\delta_0}^T\int_{\R^d}\int_{\Pi_T}\beta(v(s,y)-u(t,x))\partial_s\psi(s,y)\rho_{\delta_0}(t-s)
\varrho_{\delta}(x-y)\,dx\,dt\,dy\,ds\notag\\
&-E\int_{s=\delta_0}^T \int_0^T\int_{\R^d}\int_{\R^d}\beta(v(s,y)-u(s,x))\partial_s\psi(s,y)\varrho_{\delta}(x-y)
\rho_{\delta_0}(t-s)\,dy\,dx\,dt\,ds \Big|+\mathcal{O}(\delta_0)\notag\\
&\le E\int_{s=\delta_0}^T\int_{\R^d}\int_{\Pi_T} \Big | \beta(v(s,y)-u(t,x)) -\beta(v(s,y)-u(s,x))\Big||\partial_s\psi(s,y)|\notag\\
&\hspace{5cm}\times \varrho_{\delta}(x-y)\,\rho_{\delta_0}(t-s)\,dx\,dt\,dy\,ds +\mathcal{O}(\delta_0)\notag
\end{align}
\begin{align}
\mathcal{G}_1& \le C(\beta^{\prime})\,||\partial_s\psi||_{\infty}\, E\Big[\int_{s=\delta_0}^T\int_{\Pi_T} |u(s,x)-u(t,x)|\,
\rho_{\delta_0}(t-s)\,dx\,dt\,ds\Big] +\mathcal{O}(\delta_0)\notag \\
&\le C(\beta^{\prime})\,||\partial_s\psi||_{\infty}  E\Big[\int_{r=0}^1\int_0^T\int_{\rd}|u(t+\delta_0 r,x)-u(t,x)|\,
\rho_{1}(-r)\,dx\,dt\,dr\Big]+\mathcal{O}(\delta_0). \label{stochastic_estimate_1}
\end{align} Note that, $\lim_{\delta_0\downarrow 0 }  \int_0^T\int_{\rd}|u(t+\delta_0 r,x)-u(t,x)|\,
\,dx\,dt\rightarrow 0$ almost surely for all $ r\in [0,1] $. Therefore, by bounded convergence theorem,
$\lim_{\delta_0\downarrow 0}E\Big[\int_0^T\int_{r=0}^1\int_{\rd}|u(t+\delta_0 r,x)-u(t,x)|\,
\rho_{1}(-r)\,dx\,dr\,dt\Big]= 0$, and therefore the first step follows.
\vspace{.5cm}

\noindent{\textbf{Step 2:}}  In this step, we establish the second part of the lemma. For this, let
\begin{align*}
  \mathcal{G}_2(\eps,\delta):&= \Big| E \int_{\Pi_T}\int_{\R^d}\beta(v(s,y)-u(s,x)) \partial_s\psi(s,y)\,
  \varrho_\delta(x-y) \,dx\,dy\,ds\\
  &\hspace{1cm} -E \int_{\Pi_T}\int_{\R^d} |v(s,y)-u(s,x)| \partial_s\psi(s,y)\,
  \varrho_\delta(x-y) \,dx\,dy\,\,ds\Big|\\
  &\le ||\partial_s\psi||_{\infty}\,E\int_{\text{supp}(\psi(s,y))}\int_{\R_x^d} \int_0^T
  \Big| \beta(v(s,y)-u(s,x))- |v(s,y)-u(s,x)|\Big| \\
  &\hspace{6cm}\times\varrho_{\delta}(x-y)\,ds\,dx\,dy.
\end{align*}
As before, note that the sequence $(\beta_{\eps})_\eps$ is a sequence of functions that satisfies 
\begin{align*}
 \Big|\beta_\eps(r)-|r|\Big|\le  C \eps  ~~ \text{ for any }~~r\in \R ,
 \end{align*} we have
\begin{align}
   \mathcal{G}_2(\eps,\delta)\le& ||\partial_s\psi||_{\infty}\,\eps\, C(\psi,T).
\end{align}

Once again, let
\begin{align*}
 \mathcal{G}_3(\delta):&= \Big| E \int_{\Pi_T}\int_{\R^d} |v(s,y)-u(s,x)| \partial_s\psi(s,y)\,
  \varrho_\delta(x-y) \,dx\,dy\,\,ds\\
  &- E \int_0^T \int_{\R^d} |v(s,y)-u(s,y)| \partial_s\psi(s,y)\,
  dy\,ds\Big| \\
  & \le  E \int_{\Pi_T}\int_{\R^d}  |u(s,y)-u(s,x)| \partial_s\psi(s,y)\,
  \varrho_\delta(x-y)\,dxdy\,\,ds\\
  & \goto 0\quad \text{as}\quad \delta \goto 0. \,\, \text{(as in Lemma \ref{stochastic_lemma_1})}
\end{align*} Now

 \begin{align*}
  &\Big|E\Big[\int_{\Pi_T}\int_{x} \beta_{\eps}(v(s,y)-u(s,x))
  \partial_s \psi(s,y)\varrho_\delta(x-y)\,dx\,dy\,ds\Big] \\ &-E\Big[\int_{\Pi_T} \big|v(s,x)-u(s,x)\big|
  \partial_s \psi(s,x)\,dx\,ds\Big]\Big|\\
  \le & \mathcal{G}_2(\eps,\delta)+\mathcal{G}_3(\delta)\le  ||\partial_s\psi||_{\infty}\, C(\psi,T)\eps+\mathcal{G}_3(\delta)
  \goto \, 0~ \text{as}~ (\eps,\delta)\rightarrow (0,0) .
  \end{align*} Hence the second part follows.
\end{proof}
Next we estimate the limit of $I_3$ as $\delta_0 \goto 0$ and $ (\epsilon, \delta) \downarrow(0, 0)$.

\begin{lem}\label{stochastic_lemma_3}
\begin{align}
 \lim_{\delta_0 \goto 0} I_3= E\int_{\R^d}\int_{\Pi_T} F^\beta(v(s,y),u(s,x))\,\grad_y \psi(s,y)
 \,\varrho_{\delta}(x-y)\,dy\,ds\,dx\notag 
  \end{align} and 
   \begin{align*}
 &\lim_{(\eps, \delta)\rightarrow (0,0)} E\int_{\Pi_T}\int_{\R^d}  F^{\beta_{\eps}}(v(s,y),u(s,x))\grad_y \psi(s,y)
 \varrho_\delta(x-y)\,dx\,dy\,ds \\
 = & E\Big[\int_{\Pi_T}\sum_{k=1}^{d} \textrm{sign}(u(s,y)-v(s,y))(F_k(u(s,y))-F_k(v(s,y)))
  \partial_{y_k}\psi(s,y)\,dy\,ds\Big].
 \end{align*}
\end{lem}
\begin{proof}  The proof is divided into two steps.

\noindent{\textbf{Step 1:}}  We first verify the passage to the limit as $\delta_0 \goto 0$. 
Note that there exists $p\in N$ such that, for all $a, b,c\in \R$,
\begin{align*}
|F^{\beta}(a,b)-F^{\beta}(a,c)|\le K |b-c|(1+|b|^p+|c|^p). 
\end{align*}Therefore, upon denoting 
\begin{align}
 \mathcal{B}_1:=&\Big|E\int_{\Pi_T}\int_{\Pi_T} F^\beta(v(s,y),u(t,x))\grad_y\psi(s,y)\rho_{\delta_0}(t-s)\varrho_{\delta}(x-y)\,dy\,ds\,dx\,dt\notag \\
& -E\int_{\R^d}\int_{\Pi_T} F^\beta(v(s,y),u(s,x))\,\grad_y \psi(s,y)\,\varrho_{\delta}(x-y)\,dy\,ds\,dx\notag \Big|,
\end{align} we have 
\begin{align*}
  \mathcal{B}_1 \le& \Big|E\int_{s=\delta_0}^T \int_{\R^d}\int_{\Pi_T} F^\beta(v(s,y),u(t,x))\grad_y\psi(s,y)
  \rho_{\delta_0}(t-s)\varrho_{\delta}(x-y)\,dx\,dt\,dy\,ds\notag \\
& -E\int_{s={\delta_0}}^T\int_{\R^d}\int_{t=0}^T\int_{\R^d} F^\beta(v(s,y),u(s,x))\,\grad_y \psi(s,y)
\varrho_{\delta}(x-y)\rho_{\delta_0}(t-s)\,dx\,dt\,dy\,ds \notag \Big|+\mathcal{O}(\delta_0)\\
\le& K ||\nabla_y \psi(s,y)||_\infty E\int_{s={\delta_0}}^T\int_{\R^d}\int_{t=0}^T \big|u(s,x)-u(t,x) \big|(1+|u(s,x)|^p+|u(t,x)|^p)\\
&\hspace{6cm} \times \rho_{\delta_0}(t-s)\,dt\,dx\,ds+\mathcal{O}(\delta_0)\\
&\text{(By Cauchy-Schwartz inequality)} \\\le & C  ||\nabla_y \psi(s,y)||_\infty  \Big[ E\int_{s={\delta_0}}^T 
\int_{\R^d}\int_{t=0}^T \big|u(s,x)-u(t,x) \big|^2\rho_{\delta_0}(t-s) \,dt\,dx\,ds\Big]^{\frac 12}+\mathcal{O}(\delta_0)\\
\le & C  ||\nabla_y \psi(s,y)||_\infty  \Big[ E\int_{r={0}}^1\int_{\R^d}\int_{t=0}^T \big|u(t+\delta_0 r,x)-u(t,x)
\big|^2\rho(-r)\,dt\,dx\,dr\Big]^{\frac 12}+\mathcal{O}(\delta_0).
\end{align*} Note that, $\lim_{\delta_0\downarrow 0 }  \int_0^T\int_{\rd}|u(t+\delta_0 r,x)-u(t,x)|^2\,
\,dx\,dt\rightarrow 0$ almost surely for all $ r\in [0,1] $. Therefore, by bounded convergence theorem, $\lim_{\delta_0\downarrow 0}E\Big[\int_{r=0}^1\int_0^T\int_{\rd}|u(t+\delta_0 r,x)-u(t,x)|^2\,
\rho(-r)\,dx\,dt\,dr\Big]= 0$, and therefore the first step follows.
\vspace{.2cm}

 \noindent{\textbf{Step 2:}} 
 In this step we establish the second part of the lemma. 
Note that $F_k'(s)$ has at most polynomial growth in $ s\in\R $. It follows from direct computation that there exists
$p\in \N$ such that for all $u,v\in \R$ and $\beta = \beta_{\eps}$,
\begin{align}
\label{eq:mod_approx_1}\Big|F^{\beta_\eps}_k(v,u) -\text{sign}(u-v)((F_k(u)-F_k(v))\Big|\leq \eps C_p(1+|u|^p+|v|^p). 
\end{align} Therefore
\begin{align}
 &\Big|-E\Big[\int_{\R^d}\int_{\Pi_T} F^{\beta_{\eps}}(v(s,y),u(s,x))\,\grad_y \psi(s,y)\,\varrho_{\delta}(x-y)\,dy\,ds\,dx\notag \\
  &~~ +\int_{\R^d}\int_{\Pi_T} \sum_{k=1}^d \textrm{sign}(u(s,x)-v(s,y))(F_k(u(s,x))-F_k(v(s,y))\partial_{y_k}
  \psi(s,y)\,\varrho_{\delta}(x-y)\,dy\,ds\,dx\Big]\Big|\notag\\
 &\le E\Big[\int_{\R^d}\int_{\Pi_T}\sum_{k=1}^d\Big|F_k^{\beta_\eps}(v(s,y),u(s,x)) - \textrm{sign}(u(s,x)-v(s,y))(F_k(u(s,x))-F_k(v(s,y))\Big |\notag\\
 &\hspace{4cm} \times  |\partial_{y_k} \psi(s,y)|\,\varrho_{\delta}(x-y)\,dy\,ds\,dx\Big]\notag \\
&\le \text{Const}(\psi)\,\eps \,\Big[ 1 + \sup_{0\le s\le T}E ||v(s)||_p^p + \sup_{0\le s\le T}E ||u(s)||_p^p \Big] \notag\\
& \quad \goto 0 ~~ \text{as} \quad \eps \goto 0. 
  \end{align}
  
Since $\psi $ is smooth test function and $ F_k$'s are smooth and have polynomially growing derivatives, it is easy to verify that $F(u,v)= \mbox{sign}(u-v)(F(u)-F(v))$ is locally Lipschitz and 
\[|F(u,v)-F(\tilde{u},v)|\le C|u-\tilde{u}|(1+|u|^p+|\tilde{u}|^p).\]
Therefore, we can employ dominated convergence theorem and  conclude
\begin{align*}
 &\Big| E\Big[\int_{\Pi_T}\int_{\R^d}\sum_{k=1}^{d} \textrm{sign}(u(s,x)-v(s,y))(F_k(u(s,x))-F_k(v(s,y)))
  \partial_{y_k}\psi(s,y)\varrho_{\delta}(x-y)dx\,dy\,ds\Big] \\
  &\qquad - E\Big[\int_{\Pi_T}\sum_{k=1}^{d} \textrm{sign}(u(s,y)-v(s,y))(F_k(u(s,y))-F_k(v(s,y))
  \partial_{y_k}\psi(s,y)\,dy\,ds\Big] \Big| = \mathcal{O}(\delta). 
\end{align*} Therefore 
 \begin{align*}
 &\Big| E\int_{\Pi_T}\int_{\R^d}  F^{\beta_{\eps}}(v(s,y),u(s,x))\grad_y \psi(s,y)
 \varrho_\delta(x-y)\,dx\,dy\,ds \\
  &\qquad-E\Big[\int_{\Pi_T}\sum_{k=1}^{d} \textrm{sign}(u(s,y)-v(s,y))(F_k(u(s,y))-F_k(v(s,y)))
  \partial_{y_k}\psi(s,y)\,dy\,ds\Big]\Big| \\ &\le  \text{Const}(\psi)\,\eps + \mathcal{O}(\delta) ~\goto 
  0 ~\text{as}~~(\eps, \delta)\goto (0,0).
 \end{align*}
\end{proof}

\begin{lem}\label{stochastic_lemma_4}
It holds that 

 \begin{align} 
  \lim_{\eps\downarrow 0, \frac{\eps}{\delta}\downarrow 0, \delta\downarrow 0} \lim_{\delta_0\downarrow 0} (I_4 + I_5) =0.
   \end{align}
\end{lem}
\begin{proof}
 We can use the same reasoning as before and pass to the limit $\delta_0\downarrow 0$  and  conclude
 \begin{align}
  &\lim_{\delta_0 \goto 0}(I_4 + I_5) \notag \\
  &= E\Big[\int_{\R^d}\int_{\Pi_T}\Big\{ F^\beta(u(s,x),v(s,y))\,\grad_x \varrho_{\delta}(x-y)+ F^\beta(v(s,y),u(s,x))\notag \\
 &\hspace{4cm}  \times \,\grad_y \varrho_{\delta}(x-y)\Big\}\psi(s,y)\,dy\,ds\,dx \Big]\notag \\
  & = E\Big[\int_{\R^d}\int_{\Pi_T}\Big\{ -F^\beta(u(s,x),v(s,y))+F^\beta(v(s,y),u(s,x))\,\Big\} \cdot \grad_y \varrho_{\delta}(x-y)\psi(s,y)\,dy\,ds\,dx \Big].\notag 
 \end{align} Note that, there exists $p\in \N$, such that for all $a,b\in \R$
 \begin{align*}
 & \Big|F^{\beta_\eps}_k(a,b)-F^{\beta_\eps}_k(b,a)\Big| \\
  & \le |F_k^{\beta_\eps}(a,b)-\text{sign}(a-b)(F_k(a)-F_k(b))| + |F_k^{\beta_\eps}(b,a)-\text{sign}(b-a)(F_k(b)-F_k(a))| \\
  & \le C \eps (1+|a|^p+|b|^p).
 \end{align*} Therefore, 
 \begin{align*}
 &\Big|  E\Big[\int_{\R^d}\int_{\Pi_T}\Big\{ F^\beta(u(s,x),v(s,y))-F^\beta(v(s,y),u(s,x))\,\Big\} \cdot \grad_y \varrho_{\delta}(x-y)\psi(s,y)\,dy\,ds\,dx \Big]\Big |
 \\ 
 &\le  \eps C  E\Big[\int_{\R^d}\int_{\Pi_T} (1+|u(s,x)|^p+|v(s,y)|^p) |\nabla_y \varrho_{\delta} (x-y)|\psi(s,y)\,dy\,ds\,dx\Big]\\
&\le   \frac{\eps}{\delta} C\qquad \goto 0 \quad \text{when}\quad (\eps, \frac \eps\delta,\delta)\rightarrow (0,0,0).
 \end{align*}
 Hence the lemma follows.
\end{proof}

\begin{lem}\label{stochastic_lemma_5} The following holds:

\begin{align}
 &\lim_{\delta_{0}\goto 0} I_6 = \frac{1}{2}E \int_{\Pi_T} \int_{\R^d}\sigma^2(x,u(s,x)) \beta^{\prime\prime} (u(s,x)-v(s,y))
 \psi(s,y)\varrho_{\delta}(x-y)\,dx\,dy \,ds\label{eq:I6-delta}
\end{align} and 
\begin{align}
 \lim_{\delta_{0}\goto 0} I_7 =\frac{1}{2} E \int_{\Pi_T} \int_{\R^d}\sigma^2(y,v(s,y)) \beta^{\prime\prime} (v(s,y)-u(s,x))
 \psi(s,y)\varrho_{\delta}(x-y)\,dx\,dy \,ds.\label{eq:I7-delta}
\end{align}
\end{lem}
\begin{proof}
   We will rigorously establish \eqref{eq:I6-delta} and \eqref{eq:I7-delta}.
    Note that   
   \begin{align*}
       I_6 =&  \frac{1}{2} E \int_{\Pi_T\times \Pi_T}\sigma^2(x,u(t,x)) \beta^{\prime\prime} (u(t,x)-v(s,y))
 \psi(s,y)\rho_{\delta_0}(t-s)\varrho_{\delta}(x-y)\,dx\,dy \,ds\,dt.
   \end{align*} 
   Therefore,
   
   \begin{align*}
    & \Big| I_6 -\frac{1}{2} E \int_{\Pi_T} \int_{\R^d}\sigma^2(x,u(t,x)) \beta^{\prime\prime} (u(t,x)-v(t,y))
 \psi(t,y)\varrho_{\delta}(x-y)\,dx\,dy \,dt\Big|\\
 =  & \Big| I_6 - \frac{1}{2} E\int_{s=\delta_0}^T \int_{t=0}^T\int_{\R^d\times \R^d}\sigma^2(x,u(t,x)) \beta^{\prime\prime} (u(t,x)-v(t,y))\psi(t,y)\notag\\
 &\hspace{5cm}\times\rho_{\delta_0}(t-s)\varrho_{\delta}(x-y)\,dx\,dy \,dt\,ds \Big| + \mathcal{O}(\delta_0) \\
 \le &  E  \int_{s=\delta_0}^T \int_{t=0}^T\int_{\R^d\times \R^d}\sigma^2(x,u(t,x))\Big|\beta^{\prime\prime} (u(t,x)-v(t,y)) - \beta^{\prime\prime} (u(t,x)-v(s,y))\Big| \\
& \hspace{5cm} \times\psi(s,y)\rho_{\delta_0}(t-s)\varrho_{\delta}(x-y)\,dx\,dy \,dt\,ds\\
& + E  \int_{s=\delta_0}^T \int_{t=0}^T\int_{\R^d\times \R^d}\sigma^2(x,u(t,x)) \beta^{\prime\prime} (u(t,x)-v(t,y))|\psi(t,y)-\psi(s,y)|\notag\\
 &\hspace{6cm}\times\rho_{\delta_0}(t-s)\varrho_{\delta}(x-y)\,dx\,dy \,dt\,ds  +\mathcal{O}(\delta_0)\\
\le& ||\beta^{\prime\prime\prime}||_{\infty}E  \int_{s=\delta_0}^T \int_{t=0}^T\int_{\R^d\times \R^d}
\sigma^2(x,u(t,x)) |v(s,y)-v(t,y)|\psi(s,y)\notag\\
&\hspace{6cm} \times\rho_{\delta_0}(t-s)\varrho_{\delta}(x-y)\,dx\,dy \,dt\,ds +\mathcal{O}(\delta_0)\\
 \le& \text{Const}(\beta,\eta)E  \int_{s=\delta_0}^T \int_{t=0}^T\int_{\R^d\times \R^d}g^2(x)(1+|u(t,x)|^2) |v(s,y)-v(t,y)|\notag\\
 &\hspace{5cm}\times \psi(s,y)\rho_{\delta_0}(t-s)\varrho_{\delta}(x-y)\,dx\,dy \,dt\,ds +\mathcal{O}(\delta_0)\\
&(\text{By Cauchy-Schwartz})\\
\le&  \text{Const}(\beta,\eta)\sqrt{E  \int_{s=\delta_0}^T \int_{t=0}^T\int_{\R^d\times \R^d}g^4(x)(1+|u(t,x)|^4)\psi(s,y)
\rho_{\delta_0}(t-s)\varrho_{\delta}(x-y) \,dx\,dy \,dt\,ds}\\
 &\hspace{1cm}\times \sqrt{E  \int_{s=\delta_0}^T \int_{t=0}^T\int_{\R^d\times \R^d}|v(s,y)-v(t,y)|^2 \psi(s,y)
 \rho_{\delta_0}(t-s)\varrho_{\delta}(x-y)\,dx\,dy \,dt\,ds}+\mathcal{O}(\delta_0)\\
 \le &  \text{Const}(\beta,\eta,\psi)  \sqrt{E  \int_{s=\delta_0}^T \int_{t=0}^T\int_{\R^d}|v(s,y)-v(t,y)|^2
 \rho_{\delta_0}(t-s)\,dy \,dt\,ds}~~ +\mathcal{O}(\delta_0)\\
 \le & \text{Const}(\beta,\eta,\psi)  \sqrt{E  \int_{r=0}^1 \int_{t=0}^T\int_{\R^d}|v(t,x)-v(t+r\delta_0,x)|^2
 \rho(-r)\,dx \,dt\,dr}~~ +\mathcal{O}(\delta_0).
   \end{align*} Once again we use the fact that  $\lim_{\delta_0\downarrow 0 }  \int_0^T\int_{\rd}|u(t+\delta_0 r,x)-u(t,x)|^2\,
\,dx\,dt\rightarrow 0$. Therefore, by dominated convergence theorem,
$E  \int_{r=0}^1 \int_{t=0}^T\int_{\R^d}|v(t,y)-v(t+r\delta_0,y)|^2 \rho(-r)\,dy \,dt\,dr \rightarrow 0$ as $\delta_0\rightarrow 0$. The proof of \eqref{eq:I7-delta} is similar.

\end{proof}
\begin{lem}\label{stochastic_lemma_6} It holds that
 \begin{align}
  \lim_{\delta_0 \goto 0} I_8 =- E\int_{\Pi_T}\int_{\R^d} \sigma(x, u(t,x))\sigma(y,v(t,y))
  \beta^{\prime\prime} (u(t,x)-v(t,y)\psi(t,y)\varrho_{\delta}(x-y) \,dy\,dx\,dt.
 \end{align}
\end{lem}
\begin{proof}
Recall that
\begin{align*}
 I_8 =- E\int_{\Pi_T}\int_{\Pi_T} \sigma(x,u(t,x))\sigma(y,v(t,y)) \beta^{\prime\prime} (u(t,x)-v(t,y))
 \psi(s,y)\rho_{\delta_0}(t-s)\varrho_{\delta}(x-y) \,dy\,dx\,dt\,ds.
 \end{align*}

Therefore, as before,
  \begin{align*}
   &\Big| I_8+E\int_{\Pi_T}\int_{\R^d} \sigma(x, u(t,x))\sigma(y,v(t,y))
   \beta^{\prime\prime} (u(t,x)-v(t,y))
   \psi(t,y)\varrho_{\delta}(x-y) \,dy\,dx\,dt\Big|\\
   \le & E\int_{s=\delta_0}^T\int_{\Pi_T}\int_{\R^d} |\sigma(x, u(t,x))\sigma(y,v(t,y))| \beta^{\prime\prime}
   (u(t,x)-v(t,y))|\psi(s,y)-\psi(t,y)| \\
   &\hspace{4cm} \times\rho_{\delta_0}(t-s)\varrho_{\delta}(x-y) \,dy\,dx\,dt\,ds + \mathcal{O}(\delta_0)\\
   \le & \delta_0||\partial_t \psi||_\infty ||\beta^{\prime\prime}||_\infty  E\int_{\Pi_T}\int_{\R^d}
   |\sigma(x, u(t,x))\sigma(y,v(t,y))| \varrho_{\delta}(x-y) \,dy\,dx\,dt+ \mathcal{O}(\delta_0)\\
   \le &\delta_0||\partial_t \psi||_\infty ||\beta^{\prime\prime}||_\infty\,  C\Big(1+\sup_{0\le t\le T} E||u(t,\cdot )||^2_2+\sup_{0\le t\le T} E||v(t,\cdot)||^2_2\Big)+ \mathcal{O}(\delta_0).
   \end{align*} Hence the lemma follows by simply letting $\delta_0\downarrow 0$ in the last line.
\end{proof}
\begin{lem} \label{stochastic_lemma_7}
 Assume 
  that $ \eps\goto 0^+$, $ \delta\goto 0^+$ and $\eps^{-1}\delta^{2}\goto0^+ $, then
\begin{align}
& \limsup_{\eps\goto 0^+,~\delta\goto 0^+,~\eps^{-1}\delta^2\goto 0^+ }\lim_{\delta_0 \goto 0}\Big( I_6 +I_7 +I_8 \Big) = 0\notag
\end{align}
\end{lem}
\begin{proof}
 Since $ \beta^{\prime\prime} $ is even function,  we have 
from  Lemma \ref{stochastic_lemma_5} and Lemma \ref{stochastic_lemma_6} that
\begin{align}
&\lim_{\delta_0 \goto 0}\Big( I_6 +I_7 +I_8 \Big)\notag \\
&=\frac{1}{2}E\Big[\int_{\Pi_T}\int_{\R^d} \Big(\sigma(x,u(s,x))-\sigma(y,v(s,y))\Big)^2 \beta^{\prime\prime}(u(s,x)-v(s,y))
\psi(s,y)\varrho_\delta(x-y)\,dx\,dy\,ds\Big]\notag 
\end{align}
Now, by our assumption on $\sigma$, we have
\begin{align*}
  &\Big(\sigma(x,u(s,x))-\sigma(y,v(s,y))\Big)^2 \beta^{\prime\prime}(u(s,x)-v(s,y))\\
  & \le C\Big( |x-y|^2 + |u(s,x)-v(s,y)|^2 \Big)  \beta^{\prime\prime}(u(s,x)-v(s,y))\\
  & \le C\Big( \eps + \frac{|x-y|^2}{\eps}\Big)
  \end{align*} Therefore,
\begin{align*}
 &E \Big[\int_{(0,T]}\int_{x,y}\Big(\sigma(x,u(s,x))-\sigma(y,v(s,y))\Big)^2 \beta^{\prime\prime}(u(s,x)-v(s,y))
 \psi(s,y)\,\varrho_{\delta}(x-y)\,dx\,dy\,ds\Big]\notag \\
 \leq& C_1\Big(\eps + \eps^{-1}\delta^2\Big)T,
\end{align*} and letting $ \eps\goto 0^+ $, $ \delta\goto 0^+ $ and $\eps^{-1}\delta^{2}\goto0^+ $ 
gets us to the desired conclusion.  
\end{proof}


\begin{thm}\label{thm:L1-contraction} Assume \ref{A1}-\ref{A3}.
Suppose $u$ is a stochastic entropy solution of \eqref{eq:brown_stochconservation_laws} and $v$ is a stochastic strong entropy solution of the same equation.  Then\\
 \[E[||(u(t)-v(t))||_1]\leq E[||(u(0)-v(0))||_1 ].\] for almost every $t > 0$.
\end{thm}
\begin{proof} First we pass to the limit $ \delta_0\downarrow 0 $ in \eqref{stochas_entropy_3} and then let $\delta= \eps^{\frac 23}$
 and finally let $\eps \downarrow 0$. We use the Lemmas \ref{stochastic_lemma_1}- \ref{stochastic_lemma_7}
along with the preceding  inequality \eqref{stochas_entropy_3} and obtain
\begin{align}
 &E\Big[\int_{\R^d} |u_0(x)-v_0(x)|\psi(0,x)\,dx\Big] + E\Big[\int_{\Pi_T} |v(t,x)-u(t,x)|\partial_t\psi(t,x)
 \,dt\,dx\Big]\notag \\
 &  + E\Big[\int_{\Pi_T} F(u(t,x),v(t,x)).\grad_x \psi(t,x)\,dt\,dx\Big] \ge 0\label{stochastic_estimate_5}
\end{align}

For each $ n\in \N $, define
 \begin{align*}
\phi_n(x)=\begin{cases} 1,\quad \text{if} ~ |x|\le n\\
                       2(1-\frac{|x|}{2n}), \quad \text{if} ~n < |x|\le 2n\\
                       0,\quad \text{if}~|x|> 2n.
           \end{cases}
\end{align*}
For each $h>0 $ and fixed $t\geq 0$, define 
 \begin{align*}
\psi_h(s)=\begin{cases} 1,\quad \text{if} ~ s\le t\\
                       1-\frac{s-t}{h}, \quad \text{if} ~t \le s\le t+h\\
                       0,\quad \text{if}~s> t+h.
           \end{cases}
\end{align*}
By standard approximation, truncation and mollification argument, \eqref{stochastic_estimate_5} holds with $\psi(s,x)=\phi_n(x)
\psi_h(s)$. Define
\[ A(s)= E\Big[\int_{\R^d} |u(s,x)-v(s,x)|\,dx\Big],\]
then $ A\in L_{\text{loc}}^1([0,\infty))$. It is trivial to check that any right Lebesgue point of $A(t)$ is also a  right Lebesgue point of 
$$A_n(s)= E\Big[\int_{\R^d} \phi_n(x)|u(s,x)-v(s,x)|\,dx\Big]$$ for all $n$.
Let $t$ be a right Lebesgue point of $A$.  We choose this $t$ in the definition  of $\psi_h(s)$.
Thus, from \eqref{stochastic_estimate_5} we have
 
 \begin{align}
   &\frac{1}{h}\int_{t}^{t+h} E\Big[\int_{\R^d} |v(s,x)-u(s,x)|\phi_n(x)\,dx\Big]\,ds \notag \\
 & \le  E\Big[\int_{\Pi_{T}}  F(u(s,x),v(s,x)).\grad_x \phi_n(x)\,\psi_h(s)\,ds\,dx\Big] \notag \\
 & \hspace{2cm}+ E\Big[\int_{\R^d} |u_0(x)-v_0(x)|\phi_n(x)\,dx\Big].\notag 
 \end{align}
  Taking limit as $h\goto 0$, we obtain
   \begin{align}
   & E\Big[\int_{\R^d} |v(t,x)-u(t,x)|\phi_n(x)\,dx\Big] \notag \\
 & \le  E\Big[\int_{\R^d}\int_0^t  F(u(s,x),v(s,x)).\grad_x \phi_n(x)\,ds\,dx\Big] \notag \\
 & \hspace{2cm}+ E\Big[\int_{\R^d} |u_0(x)-v_0(x)|\phi_n(x)\,dx\Big]\notag \\
 &\le C(T) \frac{1}{n}\Big[1+ \sup_{0\le s \le T}E||u(s)||_p^p +  \sup_{0\le s \le T}E||v(s)||_p^p \Big]\notag \\
 & \hspace{2cm}+ E\Big[\int_{\R^d} |u_0(x)-v_0(x)|\phi_n(x)\,dx\Big] \label{stochastic_estimate_final}
 \end{align}
Letting $n\goto \infty$, we have from \eqref{stochastic_estimate_final}
 \[E[||(u(t)-v(t))||_1]\leq E[||(u(0)-v(0))||_1 ].\]
\end{proof}


\begin{thm}[comparison principle]\label{comparison principle}  Assume \ref{A1}-\ref{A3}. 
Suppose $u$ is a stochastic entropy solution of \eqref{eq:brown_stochconservation_laws} and $v$ is a stochastic strong entropy
solution.  Then for almost every $t>0$,
 \[E[||(u(t)-v(t))_+||_1]\leq E[||(u(0)-v(0))_+||_1 ].\]
Consequently, if  $v(0,x)\leq u(0,x)$ a.e in $x$ holds almost surely, and that  $E \Big[||(u(0,.)-v(0,.))_+||_1\Big] < \infty $, then almost surely
$ v(t,x)\leq u(t,x) $ a.e. in $ x$, and almost every $t>0$. 
\end{thm}
\begin{proof}
  The proof follows exactly same as that of  Theorem \ref{thm:L1-contraction}, if we choose $(\beta_\epsilon(r))_{\epsilon}$ to be a smooth convex approximation of $r_+=\max(0,r)$.
\end{proof}

\begin{proof}[Proof of Theorem \ref{thm:uniqueness}]

It is given that   $u$ is a stochastic entropy solution of \eqref{eq:brown_stochconservation_laws} and $v$ is a stochastic strong entropy solution and 
$\cap_{p=1,2,..} L^p(\rd)$-valued random variable  $ u_0$ satisfies
\[E \Big[||u_0||_{p}^{p} + ||u_0||_{2}^{p}\Big] < \infty , \qquad p=1,2,... ~.\]
 Therefore by  Theorem \ref{thm:L1-contraction},  as $ u(0) = v(0) $ almost surely,  we have $ u(t)=v(t) $ for almost every $t$. Hence the uniqueness is proved.

\end{proof}

\section{Vanishing viscosity and existence of entropy solutions}
In this section, we detail the mechanism of proving existence of entropic solution.  Just as the deterministic problem, here also we apply vanishing viscosity method. We must mention that a number of recent studies, including Feng and Nualart \cite{nualart:2008}, use this approach to establish existence for conservation laws driven by noise. However, this method requires rigorous wellposedness results along with a few crucial a priori estimates for the viscous problem which allows one to apply stochastic compensated compactness and get the existence. It is to be mentioned also that we need to exercise outmost caution while extracting an inviscid limit out of the vanishing viscosity approximations. The apparent inconsistencies, which are the motivations for this paper, in \cite{nualart:2008}  are largely due to the inadequacies in handling the limiting procedure. 
     
      It must be admitted here that, in \cite{nualart:2008}, the authors offer a rigorous and flawless study of the wellposedness question of viscous problem along with necessary a priori estimates.  In the first part of this section we state the relevant results without proof.
\subsection{Viscous approximation}
 Let $J\in C_c^\infty(\R)$ be the one dimensional mollifier and $\varphi\in  C_c^\infty(\R)$ be a cut-off function satisfying 
\begin{align*}
  \varphi(r)=\begin{cases}
               0\quad\text{for}\quad |r| \ge 2\\
               1 \quad \text{for}\quad |r| \le 1.
             \end{cases}
\end{align*} For $\epsilon > 0$, define the approximates $F_\epsilon(r)$ and $\sigma_\eps(x, u)$ as 

\begin{align*}
  F_\epsilon(r) &= \varphi(\epsilon |r|^2) F(r)*J_\epsilon(r)\\
 \sigma_\epsilon(x,u)& = \int_y\int_v \Big(\prod_{k=1}^{d} J_\epsilon\big(x_k-y_k\big) J_\epsilon(u-v)\Big)
\varphi(\epsilon(|y|^2+|v|^2))\sigma(y,v)\,dv\, dy,
\end{align*} and introduce the viscous perturbation of  \eqref{eq:brown_stochconservation_laws}:

 \begin{align} du_\eps(t,x) + \mbox{div}_x F_\eps(u_\eps(t,x)) \,dt &=  \sigma_\eps(x,u_\eps(t,x))d W(t) + \epsilon \Delta_{xx} u_\eps(t,x) \,dt\quad t>0, ~ x\in \R^d,
  \label{eq:stochconservation_laws-viscous}
\end{align} with the regularized initial condition

\begin{align}
 \label{regular_initial_condi} u_\eps(0,x) = \int_y \Big(\prod_{k=1}^{d} J_\epsilon\big(x_k-y_k\big)\big( u_0(y) \varphi(\epsilon |y|^2)\big) dy.
\end{align} It follows from direct computation that 

\begin{align}
  |F_{\epsilon}(r) -F(r)| \le C \epsilon (1+|r|^{2p_0})  ~\text{for some}~p_0\in \mathbb{N}\notag\\
  |\sigma_\epsilon(x, u) -\sigma(x, u)| \le C \epsilon g(x)(1+|u|).\label{eq:purturb-approx}
\end{align}

As expected, the perturbation  \eqref{eq:stochconservation_laws-viscous} are uniquely solvable and has smooth solution. We have the following proposition, a proof of which could be found in   \cite{nualart:2008}.

\begin{prop} \label{prop-moment-estimate}Let \ref{A1}-\ref{A3} hold and $\eps >0$ be a positive number. Then there is a unique $C^2(\R^d)$-valued predictable process $u_\eps(t,\cdot)$ which solves the initial value problem \eqref{eq:stochconservation_laws-viscous}-\eqref{regular_initial_condi}. Moreover,

\begin{itemize}
 \item[1.)]  For positive integers $p=2,3, 4,...$
 \begin{align}
  \sup_{\eps>0}\sup_{0\le t\le T} E\Big[ ||u_\eps(t,\cdot)||_p^p\Big] < + \infty \label{uni:moment-esti}
 \end{align}
\item[2.)]  For $\phi \in C^2(\R)$ with $ \phi,\phi^\prime, \phi^{\prime\prime}$ having at most polynomial growth
\begin{align}
 \sup_{\eps>0}E\Big[\Big|\eps \int_0^T \int_{\R^d} \phi^{\prime\prime}(u_\eps(t,x))|\grad_x u_\eps(t,x)|^2 \,dx\,dt\Big|^p\Big] < \infty
 ,~~ p=1,2,\cdots, T>0.\label{gradient-esti}
\end{align}
\end{itemize}

\end{prop}

 Our solution method relies upon being able to extract a convergent subsequence out of the family $\{u_{\epsilon}\}_{\eps > 0}$
 in an appropriate sense. However, it is needless to mention that the above moment estimates \eqref{uni:moment-esti} and 
 \eqref{gradient-esti} are not enough to ensure compactness of the family $\{u_{\epsilon}(t,x)\}$ in the classical $L^p$ sense.
 Moreover, our main emphasis is to avoid ``strong in time" framework of Feng $\&$ Nualart at any cost and we do not find it 
 appropriate to treat the family $\{u_{\epsilon}\}_{\eps > 0}$ as measure valued processes and look for convergence in the space
 of measure valued processes. This prompts us to follow \cite{vallet2012} and consider the family  $\{u_{\epsilon}\}_{\eps > 0}$ as a family of Young measures
 parametrized by $( \omega; t,x)$ and look for tightness to enable us to extract a convergence subsequence. We need to recall some of the basic features and facts about the Young measure, which is done below.

\subsection{Some basic facts about Young measures}  
Roughly speaking a Young measure is a \newline parametrized family of probability 
measures where the parameters are drawn from a $\sigma$-finite measure space. It's definition requires a $\sigma$-finite measure
space $\big(\Theta, \Sigma, \mu\big)$ and we denote by $\mathcal{P}(\R)$ the space of probability measures on $\R$. 
\begin{defi}[Young measure]
    A  Young measure from $\Theta$ into $\R$ is a map $\nu \mapsto \mathcal{P}(\R)$ such that $\nu(\cdot): \theta\mapsto \nu(\theta)(B)$ is $\Sigma$-measurable for every Borel subset $B$ of $\R$. The set of all Young measures from $\Theta$ into $\R$ is denoted by $\mathcal{R}\big(\Theta, \Sigma, \mu \big)$ or simply by  $\mathcal{R}$.
\end{defi} 
\begin{rem}
Trivially, if $u(\theta)$ is a real valued measurable function on $\big(\Theta, \Sigma, \mu\big)$ then $\nu(\theta) = \delta(\xi-u(\theta)) $ defines a Young measure on $\Theta$. In other words, with an appropriate choice of $\big(\Theta, \Sigma, \mu\big)$, the family $\{u_\eps(t,x)\}_{\eps> 0}$ can be thought of as  a family of Young measures and we are interested in finding a subsequence out of this family that `converges' to a Young measure  as $\epsilon$ goes to $0$. This obviously calls for clarification of the term {\it convergence} in this context. It turns out that  the notion of ``narrow convergence" of Young measures is the most suitable to our context.
\end{rem}

\begin{defi}[narrow convergence]
     A sequence of Young measures $\nu_n$ in $\mathcal{R}$ is said to converge $narrowly$ to $\nu$ iff for every  $A\in \Sigma$ and $h\in C_b(\R)$
     \[  \lim_{n\rightarrow \infty} \int_A\Big[\int_{\R} h(\xi)\nu_n(\theta)(d\xi)\Big]\,\mu(d\theta) = 
     \int_A\Big[\int_{\R} h(\xi)\nu(\theta)(d\xi)\Big]\,\mu(d \theta). \]
\end{defi} Next, we specify the tightness criterion for Young  measures. 

\begin{defi}[tightness]
 A family of Young measures $\{\nu_n\}_n$ in $\mathcal{R}$ is called tight if there exists an inf-compact integrand $h$ on 
 $\Theta \times \R$ such that 
         \[ \sup_{n} \int_{\Theta}\Big[\int_{\R} h(\theta ,\xi)\nu_n(\theta)(\,d\xi)\Big]\,\mu(d\theta) < \infty.\] 
\end{defi}

\begin{rem}
   Without getting into much details on the entire class of inf-compact functions, it is enough for us to know that $h(\theta, \xi) = \xi^2$ is
   one such example. With this choice of $h$ and an appropriate choice of $\big(\Theta, \Sigma, \mu\big)$, 
   by \eqref{uni:moment-esti} the family $\{u_\eps(t,x)\}_{\eps> 0}$ is tight when viewed as family of Young measures.  
\end{rem}

The tightness condition enables us to extract a subsequence from a tight family and we have the following version of 
 Prohorov's theorem to this end, a detailed proof which  could be found in \cite{Balder}.
 
 \begin{thm}\label{thm:prohorov}(1.)[Prohorov's theorem] Let $\big(\Theta, \Sigma, \mu\big)$ be  a finite measure space and 
 $\{\nu_n\}_n $ be a tight family of Young measures in $\mathcal{R}$. Then there exists a subsequence $\{\nu_{n^\prime}\}$ of
 $\{\nu_n\}_n $ and $\nu\in \mathcal{R}$ such that $\{\nu_{n^\prime}\}$ converges $narrowly$ to $\nu$. 
 
(2.) Moreover, with $\nu_n = \delta_{f_n(\theta)}(\xi)$ and given a  Caratheodory function $h(\theta, \xi)$  on $\Theta\times \R$, if $h(\theta, f_{n^\prime}(\theta))$ is uniformly integrable then
   \[  \lim_{n^\prime\rightarrow \infty} \int_{\Theta} h(\theta, f_{n^\prime}(\theta))\mu(d\theta) =   \int_\Theta\Big[\int_{\R} h(\theta,\xi)\nu(\theta)(\,d\xi)\Big]\mu(d \theta). \]
 
 \end{thm}

\subsection{The inviscid Young measure limit  of $\{u_{\epsilon}(t,x)\}_{\epsilon > 0}$}
     Let  $\mathcal{P}_T$ be the predictable $\sigma$-field on $\Omega\times (0,T)$ with respect to the filtration $\{\mathcal{F}_t\}_{0\le t \le T}$. Furthermore, we set 
     \begin{align*}
  \Theta = \Omega\times (0,T)\times \R^d,\quad \Sigma = \mathcal{P}_T \times \mathcal{L}(\R^d)\quad \text{and} \quad \mu= P\otimes \lambda_t\otimes \lambda_x,
\end{align*} where $\lambda_t$ and $\lambda_x$ are respectively the Lebesgue measures on $(0,T)$ and $\R^d$, and $\mathcal{L}(\R^d)$ be the Lebesgue $\sigma$-algebra on $\R^d$.
Note that the family $\{u_{\epsilon}(t,x)\}_{\epsilon > 0}$ could be viewed as a tight family in 
$\mathcal{R}\big(\Theta, \Sigma, \mu \big)$, but $\big(\Theta, \Sigma, \mu \big)$ is not a finite measure space.
Hence Theorem \ref{thm:prohorov} can not be readily applied to $\{u_{\epsilon}(t,x)\}_{\epsilon > 0}$.
We follow \cite{vallet2012} and get this problem with the following considerations.

 For any natural number $M$, let

\begin{align*}
  \Theta_M = \Omega\times (0,T)\times B_M,\quad \Sigma_M = \mathcal{P}_T \times \mathcal{L}(B_M)\quad \text{and} \quad \mu_M= \mu\big|_{\Theta_M},
\end{align*} where $B_M$ is the ball of radius $M$ around zero in $\R^d$ and $ \mathcal{L}(B_M)$ is the Lebesgue $\sigma$-algebra on $B_M$.  It is easily seen that $(\Theta_M, \Sigma_M, \mu_M)$ is a finite measure space and $\{u_\eps(\omega; t,x)\}_{\eps >0}$ (when restricted to $\Theta_M$) is a tight family of Young measures in $\mathcal{R}(\Theta_M, \Sigma_M, \mu_M)$. Therefore by Theorem \ref{thm:prohorov}, there exists subsequence $\eps_n \goto 0$ and $\nu^M \in \mathcal{R}(\Theta_M, \Sigma_M, \mu_M)$ such that  $\{u_{\eps_n}(\omega; t,x)\}$ converges narrowly to $\nu^M$.

 In addition, for $\bar{M}>M$, the sequence $\{u_{\eps_n}(\omega; t,x)\}$ is
  tight in $\mathcal{R}(\Theta_{\bar{M}}, \Sigma_{\bar{M}}, \mu_{\bar{M}})$, and hence admits a further subsequence, 
  say $\{u_{\eps_{n^\prime}}(\omega; t,x)\}$,
  and $\nu^{\bar{M}}\in \mathcal{R}(\Theta_{\bar{M}}, \Sigma_{\bar{M}}, \mu_{\bar{M}})$ such that 
  $\{u_{\eps_{n^\prime}}(\omega; t,x)\}$ converges narrowly to $\nu^{\bar{M}}$.  We now invoke diagonalization and conclude that there exists a subsequence  $\{u_{\eps_{n^\prime}}(\omega; t,x)\}$ with $\eps_n\goto 0$ and Young measures $\nu^{M}\in  \mathcal{R}(\Theta_{M}, \Sigma_{M}, \mu_M) $, M =1,2, 3, ...such that    $\{u_{\eps_n}(\omega; t,x)\}$ converges narrowly to $\nu^M$ in $\mathcal{R}(\Theta_{M}, \Sigma_{M}, \mu_M)$ for every $M=1,2,.....$ .  In view of  Theorem \ref{thm:prohorov}, it is easily concluded that
  
  $$\text{if} ~ \bar{M} > M~ \text{then}~\nu^M= \nu^{\bar{M}}\qquad{\mu}-a.e ~\text{on}~(\Theta_{M}, \Sigma_{M}, \mu).$$  Now define
  \begin{align}
  \label{eq:Young-extracted}\nu_{(\omega; t,x)} = \nu^M_{(\omega; t,x)}~\text{if}~(\omega; t,x) \in \Theta_M. 
  \end{align}
  Clearly, $\nu$ is well defined as an Young measure in $ \mathcal{R}(\Theta, \Sigma, \mu) $. This reasoning could now be summarized into the following lemma. 
  
\begin{lem}\label{conv-young-measure}
    Let  $\{u_\eps(t,x)\}_{\eps> 0}$ be a sequence of  $L^p(\R^d)$-valued predictable processes such that
    \eqref{uni:moment-esti} holds. Then there exists a subsequence $\{\eps_n\}$ with $\eps_n\goto 0$ and a Young measure 
    $\nu\in \mathcal{R}(\Theta, \Sigma, \mu) $ such that the following holds:

 If  $h(\theta,\xi)$ is a Caratheodory function on $\Theta\times \R$ such that $\mbox{supp}(h)\subset \Theta_M\times \R$ for
 some $M \in \mathbb{N}$ and $\{h(\theta, u_{\eps_n}(\theta)\}_n$ (where $\theta\equiv (\omega; t, x)$) is uniformly integrable,
 then 
 \[  \lim_{\eps_n\rightarrow 0} \int_{\Theta}h(\theta, u_{\eps_n}(\theta))\mu(d\theta) =   \int_\Theta\Big[\int_{\R} h(\theta, \xi)\nu(\theta)(\,d\xi)\Big]\mu(d \theta). \]
 \end{lem}
 
\begin{proof}
  The extraction of subsequence is done as described above and $\nu$ is defined in \eqref{eq:Young-extracted}. Note that if
  $M\in \mathbb{N}$ such that $\mbox{supp}(h)\subset \Theta_M\times \R$, then 
  \[  \int_{\Theta}h(\theta, u_{\eps_n}(\theta))\mu(d\theta)=  \int_{\Theta_M}h(\theta, u_{\eps_n}(\theta))\mu_M(d\theta)\]
    \[\text{and}~\int_\Theta\Big[\int_{\R} h(\theta, \xi)\nu(\theta)(\,d\xi)\Big]\mu(d \theta)=\int_{\Theta_M}\Big[\int_{\R} h(\theta, \xi)\nu^M(\theta)(\,d\xi)\Big]\mu_M(d \theta),\]
   and the convergence simply follows from Theorem \ref{thm:prohorov}.   
\end{proof} 

To this end, we intend to show that the Young measure $\nu(\theta)(\, du)$ has a point mass i.e  there is a $(\Theta, \Sigma, \mu)$-measurable function $\bar{u}$ such that for any Caratheodory function $h(\theta,\xi)$ on $\Theta\times \R$ 
       \begin{align*}
             \int_{\Theta} \Big[\int_{\R} h(\theta, \xi)\nu(\theta)(\,d\xi)\Big]\mu(d \theta) =   \int_{\Theta} h(\theta, \bar{u}(\theta)) \, \mu(\, d\theta)
       \end{align*} whenever the integrals make sense. In other words, upon writing $\theta\equiv (\omega; t,x)$ we want to 
       find out a $\mathcal{P}_T\times \mathcal{L}(\R^d)$ measurable function $\bar{u}(\omega; t, x)$ such that 
       
        \begin{align}
           \label{eq:point-mass} E\Big[ \int_{\Pi_T} \big[\int_{\R} h(\omega; t, x, \xi)\nu(\omega; t, x)(\,d\xi)\big]\,dt\, dx\Big] =   E\Big[ \int_{\Pi_T}  h(\omega; t, x,  \bar{u}(\omega; t, x)) \, dt\, dx\Big].
       \end{align} Equivalently, all that is required to be established is $ \nu(\theta)(\,d\xi) = \delta_{\bar{u}(\theta)}(\xi)\, d\xi$ for $\mu$-almost every $\theta\in \Theta$. This is  a fairly subtle point and we use idea of stochastic compensated compactness from \cite{nualart:2008} to validate this for $d=1$.
    
\subsection{Stochastic compensated compactness}  For a continuous and polynomially growing function $f:\R\rightarrow \R$, define 
      $$\overline{f}(\omega; t, x) = \int_{\xi} f(\xi)\, \nu(\omega; t,x)(\, d\xi).$$
      Then $ \overline{f}(\omega; t, x) $ is $\mathcal{P}_T \times \mathcal{L}(\R^d)$ measurable and, by \eqref{uni:moment-esti}, $\overline{f} \in L^{p}(\Theta, \Sigma, \mu) $. We further denote
      
      \begin{align*}
       \bar{u}(\omega; t, x) =   \int_{\xi} \xi\, \nu(\omega; t,x)(\, d\xi). 
      \end{align*} 
\begin{lem}\label{eq:entropy-moment-estimate} It holds that 
      \begin{align*}
                  \sup_{0\le t \le T} E\int_{\R^d} |\bar{u}(\omega; t, x)|^p \, dx < \infty
      \end{align*} for $p =2, 3, 4....$, and hence $(\omega, t)\mapsto \bar{u}(\omega; t, x)$ is a $\mathcal{P}_T$-measurable and $L^2(\R^d)$-valued process. 
\end{lem}
\begin{proof}
 Let $g$ be a Lebesgue measurable function on $(0,T)$ and $ g\in L^1\big((0,T)\big) $. Then, for every $M\in \mathbb{N}$, by Lemma \ref{conv-young-measure}, 
 
 \begin{align*}
      \int_0^T g(t)E\int_{B_M} |\bar{u}(\omega; t, x)|^p \, dx \,dt
       &\le E\int_0^T\int_{B_M}[\int_{\xi} g(t) |\xi|^p\nu(\omega; t, x)(\, d\xi)]\, dx\, dt\\
       & = \lim_{\epsilon_n \rightarrow 0} E\int_0^T\int_{B_M}g(t)[\int_{\xi} |\xi|^p\delta_{u_{\epsilon_n}(\omega; t, x)}(\, d\xi)]\, dx\, dt\\
        & = \lim_{\epsilon_n \rightarrow 0} E\int_0^T\int_{B_M}g(t)|u_{\epsilon_n}(t,x)|^p\, dx\, dt\\
        & \le ||g||_{L^1\big((0,T)\big)} \sup_{\epsilon}\sup_{0\le t \le T} E \big[||u_\epsilon(t,\cdot)||_p^p\big].
 \end{align*} Note that the last line is independent of $M$, therefore by letting $M$ to infinity in the first  expression we have
 \begin{align*}
    \int_0^T g(t)E\int_{\R^d} |\bar{u}(\omega; t, x)|^p \, dx \,dt \le ||g||_{L^1\big((0,T)\big)} \sup_{\epsilon}\sup_{0\le t \le T} E \big[||u_\epsilon(t,\cdot)||_p^p\big]
 \end{align*} for all $g \in L^1\big((0,T)\big)$, which implies that $E||\bar{u}(t, \cdot)||_p^p\in L^{\infty}\big((0,T)\big)$.
  
\end{proof}
         
 Next we state the main result of this subsection. 
 \begin{lem}\label{lem:singular_support} Assume that $d=1$ and  \ref{A1}-\ref{A3} holds. Then it holds that
\begin{align}
\label{point-mass-1}F(\bar{u}(\theta))\, \mu(\,d\theta)=  \big[\int_{\xi\in \R} F(\xi)\nu(\theta)(d\xi)\big] \, \mu(\, d\theta).
\end{align}
In addition, if \ref{A4} holds, then
\begin{align}
\label{point-mass-2} \nu(\theta)(du) \, \mu(d\theta)=\delta_{\bar{u}(\theta)}(du)\, \mu(d\theta).
\end{align}
 \end{lem}
 \begin{rem}
     If we expand our notation and write $\theta = (\omega; t, x)$, then \eqref{point-mass-1} simply means that for any $\Sigma$-measurable function $h((\omega; t, x))$ on $\Theta$, it holds that
       \begin{align*}
            \int_{\Omega} \int_{\Pi_T} \Big[\int_{\xi} h((\omega; t, x)) F(\xi)\nu(\omega; t, x)(\,d\xi)\Big]\, dx\, dt\, dP(\omega)=  E \int_{\Pi_T} h((\omega; t, x)) F( \bar{u}(\omega; t, x)) \, dx\, dt,
       \end{align*} provided the integrals make sense.  Similarly, \eqref{point-mass-2} means that for any given Caratheodory function  $h((\omega; t, x), \xi)$ on $\Theta\times \R$, one has 
       \begin{align*}
           \int_{\Omega} \int_{\Pi_T} \Big[\int_{\xi} h((\omega; t, x), \xi) \nu(\omega; t, x)(\,d\xi)\Big]\, dx\, dt\, dP(\omega)=  E \int_{\Pi_T} h((\omega; t, x), \bar{u}(\omega; t, x)) \, dx\, dt.
       \end{align*}
 \end{rem}
  The proof of Lemma \ref{lem:singular_support} requires the application of a stochastic version of div-curl lemma, and \cite[Theorem A.2]{nualart:2008} is such a version. Let us also mention that we find the proof of \cite[Theorem A.2]{nualart:2008} to be absolutely flawless and will be using it here too. The next lemma is an important technical step to prove  Lemma \ref{lem:singular_support}.
\begin{lem}\label{div-curl}
Let $(\Phi_i,\Psi_i), i=1,2$  be two choices of entropy flux pairs, where
$\Phi_i$'s  have  at most polynomial growth (therefore $\Psi_i$ will have
at most polynomial growth as well).
 For every deterministic $\psi\in C_c^{\infty}(\Pi_T) $,
\begin{align}
 \langle \psi,\overline{\Psi_1\Phi_2}-\overline{\Phi_1\Psi_2}\rangle\overset{\text{D}}
{=}\langle
\psi,\overline{\Psi_1}\cdot \overline{\Phi_2}-\overline{\Phi_1}\cdot\overline{\Psi_2}\rangle \label{eq: weakcon_1}
\end{align}
\end{lem} 

\begin{proof}
 Let $\psi \in C_c^{\infty}(\Pi_T) $ and $B \in \mathcal{F}_T$. Define 
 \begin{align}
  X_\eps(\omega):=\int_{\Pi_T} \psi(t,x)\Big(\Psi_1(u_\eps(t,x))\Phi_2(u_\eps(t,x))-\Phi_1(u_\eps(t,x))\Psi_2(u_\eps(t,x))\Big)\,dx\,dt.
 \end{align} Note that, by martingale representation theorem, there exists a continuous martingale $Z_t$ such that $Z_T = {\bf 1}_{B}$. 
 Then
\begin{align}
& \lim_{\eps_n\goto 0^+} E \big[{\bf1}_{B}(\omega) X_{\eps_n}(\omega)\Big] \notag \\
& = \lim_{\eps_n\goto 0^+} \int_{\Pi_T}E\Big[ E\big [{\bf1}_{B}(\omega)| \mathcal{F}_t\big]
\psi(t,x) \big(\Psi_1(u_{\eps_n}(t,x))\Phi_2(u_{\eps_n}(t,x))-\Phi_1(u_{\eps_n}(t,x))\Psi_2(u_{\eps_n}(t,x))\big)\Big]\,dx\,dt \notag\\
& = \lim_{\eps_n\goto 0^+} \int_{\Pi_T}E\Big[ Z_t\,
\psi(t,x) \big(\Psi_1(u_{\eps_n}(t,x))\Phi_2(u_{\eps_n}(t,x))-\Phi_1(u_{\eps_n}(t,x))\Psi_2(u_{\eps_n}(t,x))\big)\Big]\,dx\,dt\notag\\
&\big[\text{By Lemma \ref{conv-young-measure}}\big]\notag\\&=  \int_{\Pi_T} E\Big[ Z_t\,
\psi(t,x) \Big( \int_{u}\big(\Psi_1(u)\Phi_2(u)-\Phi_1(u)\Psi_2(u))\big)\nu(\omega; t, x)(du)\Big)\Big]\,dx\,dt\notag \\
&= \int_{\Omega} \int_{\Pi_T} {\bf1}_{B}(\omega)
\psi(t,x) \Big( \int_{u}\big(\Psi_1(u)\Phi_2(u)-\Phi_1(u)\Psi_2(u))\big)\nu(\omega; t,x)(du)\Big)\,dx\,dt\,dP(\omega)\notag \\
&= \int_{\Omega} \int_{\Pi_T} {\bf1}_{B}(\omega) \psi(t,x) \Big(\overline{\Psi_1\Phi_2}(\theta)
- \overline{\Phi_1\Psi_2}(\theta)\Big)\,dt\,dx \,dP(\omega) \notag\\
& = \int_{\Omega} {\bf1}_{B}(\omega) \langle \psi,\overline{\Psi_1\Phi_2} -\overline{\Phi_1\Psi_2}\rangle (\omega)\,dP(\omega) \notag \\
&\equiv \int_{\Omega} {\bf1}_{B}(\omega) X(\omega)\,dP(\omega) \label{eq:martingale-representation}
\end{align} where $X(\omega)=\langle \psi,\overline{\Psi_1\Phi_2} -\overline{\Phi_1\Psi_2}\rangle (\omega)$. This implies that
$X_{\eps_n} \overset{\text{a.s}} {\longrightarrow} X$ and hence $X_{\eps_n} \overset{\text{D}} {\rightarrow} X$. In other words
\begin{align}
 \lim_{\eps_n\goto 0^+} \int_{\Pi_T}
\psi(t,x)\Big(\Psi_1(u_{\eps_n}(t,x))\Phi_2(u_{\eps_n}(t,x))-\Phi_1(u_{\eps_n}(t,x))\Psi_2(u_{\eps_n}(t,x))\Big)\,dx\,dt
 \overset{\text{D}} {=}\langle\psi, \overline{\Psi_1\Phi_2} -\overline{\Phi_1\Psi_2}\rangle. \notag
\end{align}

Let $G_\eps(t,x)=(\Phi_1(u_\eps(t,x)),\Psi_1(u_\eps(t,x))) $ and $H_\eps(t,x)=(-\Psi_2(u_\eps(t,x)),\Phi_2(u_\eps(t,x))) $. 
By the moment estimate \eqref{uni:moment-esti},  we see that  the families $\{G_\eps\}_{\eps > 0}$ and $\{H_\eps\}_{\eps > 0}$ are stochastically bounded as $
L^2(\Pi_T;\R^2) $-valued random variables. 

We now call upon \cite[Lemma 4.18]{nualart:2008} and claim that $\{\partial_t\Phi_{\eps_n}^i +\partial_x\Psi_{\eps_n}^i \}_n$, where $\Phi_\eps^i= \Phi_i(u_\eps(\cdot,\cdot))$ and $\Psi_\eps^i= \Psi(u_\eps(\cdot,\cdot))$ and i=1, 2;  are tight sequences as $H^{-1}\Big(\Pi_T\Big)$-valued random variables.
 This means both $ \{\grad \cdot G_{\eps_n}  \}_n$ and $\{\grad \times H_{\eps_n} \}_n$ are tight as sequences of 
 $H^{-1}(\Pi_T) $-valued random variables. In view of Lemma \ref{conv-young-measure},
 we see that condition $(2)$ of the  div-curl lemma [ \cite{nualart:2008}, Theorem A.2 ] holds. 
 Therefore, one can apply  the div-curl lemma and have
\begin{align}
 \lim_{\eps_n\goto 0^+} \int_{\Pi_T}
\psi(t,x)\Big(\Psi_1(u_{\eps_n}(t,x))\Phi_2(u_{\eps_n}(t,x))-\Phi_1(u_{\eps_n}(t,x))\Psi_2(u_{\eps_n}(t,x))\Big)\,dx\,dt
 \overset{\text{D}} {=}\langle\psi,\overline{\Psi_1}.\overline{\Phi_2}-\overline{\Phi_1}.\overline{\Psi_2}\rangle.\notag
\end{align}
Thus, for  every deterministic $\psi\in C_c^{\infty}(\Pi_T) $,
\begin{align}
 \langle\psi,\overline{\Psi_1\Phi_2}-\overline{\Phi_1\Psi_2}\rangle\overset{\text{D}}
{=}\langle\psi,\overline{\Psi_1}.\overline{\Phi_2}-\overline{\Phi_1} .\overline{\Psi_2}\rangle.\notag 
\end{align}

\end{proof}

\begin{proof}[Proof of Lemma \ref{eq:entropy-moment-estimate}] 
 Let $\phi \in C_c^{\infty}(\Pi_T)$ be nonnegative deterministic test function. Choose $\Phi_1(u) = u$ and $\Psi(u)= F(u)\equiv \Psi_1(u)$.  Then $\Psi_2(u) = \int_0^u \big(F^{\prime}(r)\big)^2 \, dr$. Now apply Lemma \ref{div-curl} and arrive at \
 \begin{align}
 \langle \psi,\overline{F^2}-\overline{u\Psi_2}\rangle\overset{\text{D}} {=}\langle
\psi,(\bar{F})^2-\bar{u}.\overline{\Psi_2}\rangle. \label{eq: weakcon_2}
\end{align} Note that, by Schwartz inequality , for any $u \in \R$
\begin{align}
 \Big(F(u)-F(\bar{u}(\theta))\Big)^2 &= \Big(\int_{\bar{u}(\theta)}^{u} F'(v)\,dv\Big)^2
 \leq ( u-\bar{u}(\theta))\Big(\Psi_2(u)-\Psi_2(\bar{u}(\theta))\Big). \label{schwarz1}
\end{align} Integrating the inequality \eqref{schwarz1} against $\nu(\theta)(\, du)$, we have 
\begin{align}
\label{eq: weakcon_2.1}\overline{F^2}(\theta)-2\bar{F}(\theta)F(\bar{u}(\theta)) +( F(\bar{u}(\theta)))^2 \leq
\overline{u\Psi_2}(\theta)-\bar{u}(\theta).\overline{\Psi_2}(\theta). 
\end{align} We now multiply \eqref{eq: weakcon_2.1} by $\psi(t,x)$ and integrate against $\mu(\, d\theta)$ (i.e $\,dx \,dt \, dP(\omega)$) and obtain 
\begin{align*}
 \int_{\Pi_T} \psi(t,x) E[(\bar{F}-F(\bar{u}))^2]\,dt\,dx \leq \int_{\Pi_T}
\psi(t,x) E[(\overline{u\Psi_2}-\bar{u}.\overline{\Psi_2})-(\overline{F^2}-(\bar{F})^2)]\,dt\,dx = 0 \quad\text{by~ \eqref{eq: weakcon_2}}.
\end{align*} In other words 
 \begin{align}
	 \int_{\Pi_T} \psi(t,x) E[(\overline{F}-F(\bar{u}))^2]\,dt\,dx = 0,
 \end{align} which implies 
 
 \begin{align}
\label{eq:point-4-1}\int_{\Pi_T} \psi(t,x) \overline{F}(\omega; t,x) \,dx\, dt = \int_{\Pi_T} \psi(t,x) F(\bar{u}(\omega; t,x))\,dx\,dt\qquad \text{almost surely}.
 \end{align} In view of \eqref{eq:point-4-1} and \eqref{eq: weakcon_2},  one has
 
 \begin{align}
 & \int_{\Pi_T}\psi(t,x) E\Big[ \int_{u\in\R}
\Big(F(u)-F(\bar{u}(\omega; t,x))\Big)^2 \nu(\omega; t,x)(du)\Big] \,dt\,dx \notag\\
&= \int_{\Pi_T}\psi(t,x) E[\overline{F^2}-(\bar{F})^2]\,dt\,dx \notag \\
&=\int_{\Pi_T}\psi(t,x) E[(\overline{u\Psi_2}-\bar{u}.\overline{\Psi_2})]\,dt\,dx \notag \\
&= \int_{\Pi_T}\psi(t,x) E\Big[\int_{u\in\R}
\Big((u-\bar{u}(\omega; t,x))(\Psi_2(u)-\Psi_2(\bar{u}(\omega; t,x)))\Big) \nu(\omega; t,x)(du)\Big]\,dt\,dx. \notag
\end{align} We now invoke \eqref{schwarz1} and arbitrariness of $\psi$ to conclude that for $\mu$-almost all $\theta \in \Theta$ and every $u\in \text{support}(\nu(\theta))$, we must have 
 \begin{align}
\Big(F(u)-F(\bar{u}(\theta))\Big)^2=(u-\bar{u}(\theta))\Big(\Psi_2(u)-\Psi_2(\bar{u}(\theta))\Big).\label{schwarz2}
\end{align} To this end we recall the condition for equality in Schwartz inequality and conclude that \eqref{schwarz2} is
possible only if $F^{\prime}$ is constant between $u$ and $\bar{u}(\theta)$. This is an impossibility if $u$ is different
from $\bar{u}(\theta)$, thanks to \ref{A4}. Therefore $\nu(\theta)$ is a probability measure on $\R$ which is supported at 
the point $\bar{u}(\theta)$ for $\mu$-almost every $\theta \in \Theta$. In other words, \eqref{point-mass-2} holds.
 
\end{proof}

\subsection{Existence of entropy solution.} In view of the results and analysis above, it is now routine to show that
$\bar{u}(\omega; t,x)$ satisfies the stochastic entropy condition. From here onwards, will simply drop $\omega$ and write
$\bar{u}(t,x)$ in place of $\bar{u}(\omega; t,x)$.  We begin by fixing a non negative test function
$ \psi\in C_c^\infty([0, \infty)\times \R)$,  $B\in \mathcal{F}_T$ and
 convex entropy pair $(\beta,\zeta)$.  

Assume that $\zeta_\eps$ be the entropy flux with flux function $F_\epsilon$  which would approximate $\zeta$.  Now apply It\^{o}'s formula on \eqref{eq:stochconservation_laws-viscous} followed by It\^{o} product rule (as in \eqref{eq:entropy_derivation}) and then multiply by $\psi(t,x) {\bf 1}_{B}$ and integrate to obtain
 \begin{align}
0\le &E\Big[{\bf 1}_{B}\int_{\R} \beta(u_0^{\eps_n}(x))\psi(0,x)\,dx\Big]
-{\eps} E\Big[{\bf 1}_{B}\int_{\Pi_T}\beta^\prime(u_{\eps_n}(t,x))\grad u_{\eps_n}(t,x)\cdot \grad \psi(t,x)\,dx\,dt\Big] \notag \\
&+E\Big[{\bf 1}_{B} \int_{\Pi_T} \Big(\beta(u_{\eps_n}(t,x))\partial_t \psi(t,x)+\zeta_{\eps_n}(u_{\eps}(t,x))\cdot\grad\psi(t,x)\Big) \,dt\,dx\Big] \notag\\
& + E\Big[{\bf 1}_{B}\int_{\Pi_T} \sigma_{\eps_n}(x,u_{\eps_n}(t,x)) \beta^\prime(u_{\eps_n}(t,x))\psi(t,x)\,dx\,dW(t)\Big]\notag\\
& + \frac{1}{2} E\Big[{\bf 1}_{B}\int_{\Pi_T} \sigma^2_{\eps_n}(x,u_{\eps_n}(t,x)) \beta^{\prime \prime}(u_{\eps_n}(t,x))\psi(t,x)\,dx\,dt\Big]
\label{viscous-measure-inequality}
\end{align} 

With the help of uniform moment estimates and \eqref{eq:purturb-approx}; \eqref{viscous-measure-inequality} gives 

   \begin{align}
0\le &E\Big[{\bf 1}_{B}\int_{\R} \beta(u_0^{\eps_n}(x))\psi(0,x)\,dx\Big]
-{\eps} E\Big[{\bf 1}_{B}\int_{\Pi_T}\beta^\prime(u_{\eps_n}(t,x))\grad u_{\eps_n}(t,x) \cdot \grad \psi(t,x)\,dx\,dt\Big] \notag \\
&+E\Big[{\bf 1}_{B} \int_{\Pi_T} \Big(\beta(u_{\eps_n}(t,x))\partial_t \psi(t,x)+\zeta(u_{\eps_n}(t,x)) \cdot\grad\psi(t,x)\Big) \,dt\,dx\Big] \notag\\
& + E\Big[{\bf 1}_{B}\int_{\Pi_T} \sigma(x,u_{\eps_n}(t,x)) \beta^\prime(u_{\eps_n}(t,x))\psi(t,x)\,dx\,dW(t)\Big]\notag\\
& + \frac{1}{2} E\Big[{\bf 1}_{B}\int_{\Pi_T} \sigma^2(x,u_{\eps_n}(t,x)) \beta^{\prime \prime}(u_{\eps_n}(t,x))\psi(t,x)\,dx\,dt\Big]
+\mathcal{O}(\eps_n)\label{viscous-measure-inequality-2}
\end{align} 

  All that is left now is to justify passage to the limit $\epsilon_n\rightarrow 0$ in \eqref{viscous-measure-inequality-2}. In view of the estimate \eqref{gradient-esti}, it holds that 
\begin{align}
 \label{eq:passage-1} \lim_{\eps_n \rightarrow 0 } {\eps_n} E\Big[{\bf 1}_{B}\int_{\Pi_T}\beta^\prime(u_{\eps_n}(t,x))\grad u_{\eps_n}(t,x)\cdot\grad \psi(t,x)\,dx\,dt\Big] = 0.
\end{align} Furthermore, it follows from straightforward computation that

\begin{align}
 \label{eq:passage-2} \lim_{\eps_n \rightarrow 0 } E\Big[{\bf 1}_{B}\int_{\R} \beta(u_0^{\eps_n}(x))\psi(0,x)\,dx\Big]
 = E\Big[{\bf 1}_{B}\int_{\R} \beta(u_0(x))\psi(0,x)\,dx\Big]. 
\end{align} Note that  ${\bf 1}_{B}(\omega)$ may not be $\Sigma$-measurable, but we can adapt  the technique as in the derivation of \eqref{eq:martingale-representation} in the proof of Lemma \ref{div-curl} and apply Lemmas \ref{conv-young-measure} and \ref{lem:singular_support} to have
\begin{align}
& \lim_{\eps_n \rightarrow 0 }E\Big[{\bf 1}_{B} \int_{\Pi_T} \Big(\beta(u_{\eps_n}(t,x))\partial_t \psi(t,x)+\zeta(u_{\eps_n}(t,x)) \cdot\grad\psi(t,x)\Big) \,dt\,dx\Big] \notag \\
  =& E\Big[{\bf 1}_{B} \int_{\Pi_T} \Big(\beta( \bar{u}(t,x))\partial_t \psi(t,x)+\zeta( \bar{u}(t,x)) \cdot \grad\psi(t,x)\Big)\,dt\,dx\Big] \label{eq:passage-3}
\end{align} and 

\begin{align}
& \lim_{\eps_n\rightarrow 0 } \frac{1}{2} E\Big[{\bf 1}_{B}\int_{\Pi_T} \sigma^2(x,u_{\eps_n}(t,x)) \beta^{\prime \prime}(u_{\eps_n}(t,x))\psi(t,x)\,dx\,dt\Big] \notag \\
= &\frac{1}{2} E\Big[{\bf 1}_{B}\int_{\Pi_T} \sigma^2(x, \bar{u}(t,x)) \beta^{\prime \prime}(\bar{u}(t,x))\psi(t,x)\,dx\,dt\Big]. \label{eq:passage-4}
\end{align}
 Now passage to the limit in the martingale term requires some additional reasoning.  Let $\Gamma = \Omega\times [0,T] $, $\mathcal{G}= \mathcal{P}_T$ and
 $\varsigma = P\otimes \lambda_t $. 
The space   $L^2\big((\Gamma, \mathcal{G}, \varsigma); \R\big)$ represents the space of square integrable predictable integrands
for It\^{o} integrals with respect to $W(t)$.  Moreover, by It\^{o} isometry and martingale representation theorem,
it follows that It\^{o} integral defines isometry between two Hilbert spaces $L^2\big((\Gamma, \mathcal{G}, \varsigma); \R\big)$  and $L^2\big((\Omega, \mathcal{F}_T); \R\big)$. In other words, if $\mathcal{I}$ denotes the It\^{o} integral operator and $\{X_n\}_n$ be sequence in $L^2\big((\Gamma, \mathcal{G}, \varsigma); \R\big)$ weakly converging to $X$; then $\mathcal{I}(X_n)$ will converge weakly to $\mathcal{I}(X)$ in $L^2\big((\Omega, \mathcal{F}_T); \R\big)$.

 We again apply  Lemmas \ref{conv-young-measure} and \ref{lem:singular_support}  and conclude that for any  $h(t) \in L^2\big((\Gamma, \mathcal{G}, \varsigma); \R\big) $

   \begin{align*}
  & \lim_{\eps_n \rightarrow 0}E\Big[ \int_0^T \int_{\R} \sigma(x,u_{\eps_n}(t,x)) \beta^\prime(u_{\eps_n}(t,x)) h(t)\psi(t,x)\,dx\,dt \Big] \notag \\
  =&  E \Big[ \int_0^T \int_{\R} \sigma(x, \bar{u}(t,x)) \beta^\prime(\bar{u}(t,x)) h(t)\psi(t,x)\,dx\,dt. \Big]\notag
 \end{align*}  Hence, if we denote $X_n = \int_{\R} \sigma(x, u_{\eps_n}(t,x)) \beta^\prime(u_{\eps_n}(t,x)) \psi(t,x)\,dx $ and\newline $ X =  \int_{\R} \sigma(x, \bar{u}(t,x)) \beta^\prime(\bar{u}(t,x)) h(t)\psi(t,x)\,dx$, then $X_n$ converges weakly to $X$ in $L^2\big((\Gamma, \mathcal{G}, \varsigma); \R\big) $. Therefore $\mathcal{I}(X_n)$ will converge weakly to $\mathcal{I}(X)$ in $L^2\big((\Omega, \mathcal{F}_T); \R\big)$. In other words, the following lemma holds. 

\begin{lem}\label{lem:martingale} For every $B\in \mathcal{F}_T$
\begin{align*}
  &  \lim_{\eps_n \rightarrow 0} E\Big[{\bf 1}_{B} \int_0^T \int_{\R}  \sigma(x,u_{\eps_n}(t,x)) \beta^\prime(u_{\eps_n}(t,x))\psi(t,x)\,dx\,dW(t) \Big]\notag \\
  =&  E \Big[{\bf 1}_{B} \int_0^T \int_{\R} \sigma(x, \bar{u}(t,x)) \beta^\prime(\bar{u}(t,x))\psi(t,x)\,dx\,dW(t) \Big]
\end{align*} 
\end{lem}

 Now simply combine \eqref{eq:passage-1}-\eqref{eq:passage-4} along with Lemma \ref{lem:martingale} and pass to the limit $\eps_n \downarrow 0$ in 
\eqref{viscous-measure-inequality-2} and obtain

\begin{align}
0\le &E\Big[{\bf 1}_{B}\int_{\R} \beta(u_0 (x))\psi(0,x)\,dx\Big]+ E \Big[{\bf 1}_{B} \int_{\Pi_T} \beta( \bar{u}(t,x))\partial_t \psi(t,x)\,dt\,dx\Big]\notag \\
 + &  E \Big[{\bf 1}_{B} \int_{\Pi_T}\zeta( \bar{u}(t,x)) \cdot \grad\psi(t,x)\,dt\,dx\Big] 
   + \frac{1}{2} E\Big[{\bf 1}_{B}\int_{\Pi_T} \sigma^2(x, \bar{u}(t,x)) \beta^{\prime \prime}(\bar{u}(t,x))\psi(t,x)\,dx\,dt\Big]\notag \\
 &+ E \Big[{\bf 1}_{B} \int_0^T \int_{\R} \sigma(x, \bar{u}(t,x)) \beta^\prime(\bar{u}(t,x))\psi(t,x)\,dx\,dW(t) \Big] 
\label{viscous-measure_1-inequality}
\end{align} 

Finally, we now combine the  results  above and claim that $\bar{u}(t,x)$ is a stochastic entropy solution of \eqref{eq:brown_stochconservation_laws}.

\begin{lem}\label{lem:entropysolution}
 The function $\bar{u}(t,x)$ is an entropy solution of \eqref{eq:brown_stochconservation_laws}. 
\end{lem}
\begin{proof} The predictability of $\bar{u}(t,x)$ and necessary moment estimates are derived in Lemma \ref{eq:entropy-moment-estimate}.  Since \eqref{viscous-measure_1-inequality} is satisfied for all $B\in \mathcal{F}_T$, we must have 

\begin{align}
 &\int_{\R} \beta(u_0 (x))\psi(0,x)\,dx+  \int_{\Pi_T} \beta( \bar{u}(t,x))\partial_t \psi(t,x)\,dt\,dx\notag \\
  & +  \int_{\Pi_T}\zeta( \bar{u}(t,x)) \cdot \grad\psi(t,x)\,dt\,dx
   + \frac{1}{2} \int_{\Pi_T} \sigma^2(x, \bar{u}(t,x)) \beta^{\prime \prime}(\bar{u}(t,x))\psi(t,x)\,dx\,dt\notag \\
 &+  \int_0^T \int_{\R} \sigma(x, \bar{u}(t,x)) \beta^\prime(\bar{u}(t,x))\psi(t,x)\,dx\,dW(t)
  \ge 0\quad P- \text{a.s.}\notag 
\end{align} In other words, $\bar{u}$ satisfies the stochastic entropy condition.

\end{proof} 
 
 
 \section{existence of strong entropy solution}
  In this section we establish that
the vanishing viscosity limit $v(t,x) = \bar{u}(t,x)$ is indeed a strong entropy solution. To this end, let $\tilde{u}(t) = \tilde{u}(t, x)$ be an $\mathcal{F}_t$-predictable and $L^2(\R)$-valued process with 

\begin{align}
\sup_{0\le t\le T}  E \Big[||\tilde{u}(t)||_p^p \Big] < \infty, \quad \text{for all} \quad T > 0,~p=2, 4,... \label{eq:exist-strong-entropy}
\end{align}
Furthermore, let $\beta$ be a smooth convex function approximating the absolute value in $\R$ and
$\psi\in C_c^\infty([0, \infty)\times \R)$ be a nonnegative test function. For constants $\delta >0, ~\delta_0 >0$, define
\begin{align*}
    \phi_{\delta,\delta_0}(t,x,s,y) = \rho_{\delta_0}(t-s)\varrho_{\delta}(x-y) \psi(s,y).
\end{align*}
\begin{lem}\label{lem:strong-entropy-condition}
   For each $T > 0$, there exists a deterministic function $A(\delta, \delta_0)$ such that 
   
\begin{align*}
 &E\Big[\int_0^T \int_y\int_0^T\int_x \sigma(x,\tilde{u}(r,x))\beta^\prime(\tilde{u}(r,x)-v)
\phi_{\delta,\delta_0}(r,x,s,y)
 \,dx \,dW(r)\Big|_{v=v(s,y)}\,dy\,ds\Big] \notag\\
\le& - E\Big[\int_{\Pi_T}\int_{\Pi_T}  \sigma(x,\tilde{u}(r,x))\sigma(y,v(r,y))\beta^{\prime\prime}(\tilde{u}(r,x)-v(r,y))\notag \\
&\hspace{2cm}\times \phi_{\delta,\delta_0}(r,x,s,y)
 \,dr\,dx \,dy \,ds \Big] + A(\delta,\delta_0).
\end{align*}  
Furthermore, for fixed $\delta$, $\psi$ and $\beta$, the function $A(\delta,\delta_0)$ has the property that 
\begin{align*}
\lim_{\delta_0\rightarrow 0} A(\delta,\delta_0) = 0.
\end{align*}
\end{lem}
 A significant part of the proof is built on ideas borrowed from \cite{nualart:2008}, and the proof requires some preparation. Given a nonnegative test function $\phi\in C_c^\infty(\Pi_{\infty}\times \Pi_\infty)$ and $\beta \in C^\infty(\R)$ such that $\beta^\prime, 
\beta^{\prime\prime}\in C_b(\R)$, define

\begin{align*}
 J[\beta, \phi](s; y, v) := \int_0^T\int_x \sigma(x,\tilde{u}(r,x))\beta(\tilde{u}(r,x)-v) \phi(r,x,s,y)
 dx\,dW(r) 
\end{align*}
where  $0\le s\le T$, $(y,v)\in \R\times \R$. 

 Since the test function $\psi$ has compact support, there exists $c_\phi > 0$ such that $ J[\beta, \phi](s; y, v)= 0$ if $y > c_\phi$ and $0\le s\le T$ .
\begin{lem}\label{lem:differentiation} The following identities hold:
   \begin{align*}
    &\partial_v  J[\beta, \phi](s; y, v) =  J[-\beta^\prime, \phi](s; y, v)\\
    & \partial_{y}  J[\beta, \phi](s; y, v) =  J[\beta, \partial_{y}\phi](s; y, v).
   \end{align*}

\end{lem}
\begin{proof}
  The proof is similar to the that  of  Leibniz integral rule.
\end{proof}

\begin{lem} \label{lem:L-infinity estimate}
 Let $\beta\in C^\infty(\R)$ be function such that $\beta^\prime, \beta^{\prime\prime}\in C_c^\infty(\R) $.
 Then there exists a constant $C=C(\beta', \psi )$  such that
 \begin{align}
 \sup_{0\le s\le T}\Big( E|| J[\beta,\phi_{\delta,\delta_0}](s;\cdot,\cdot)||_{L^\infty(\R\times \R)}^2\Big) \le
  \frac{C(\beta', \psi )}{\delta_0^\frac{3}{2}}.\label{eq:l-infinity bound}
 \end{align}
\end{lem}
\begin{proof} We intend to establish \eqref{eq:l-infinity bound} with the help of appropriate Sobolev embedding theorem.  To this end, we begin with 
 \begin{align}
 &E\Big[|| J[\beta,\phi_{\delta,\delta_0}](s;\cdot,\cdot)||_4^4\Big]\notag \\
  & =E\Big[\int_v \int_y \Big|J[\beta,\phi_{\delta,\delta_0}](s;y,v)\Big|^4 \,dy\,dv\Big] \notag \\
  & =  E\Big[\int_v \int_y \Big|\int_{0}^T \int_{x} \sigma(x,\tilde{u}(r,x)) \beta(\tilde{u}(r,x)
  -v )\rho_{\delta_0}(r-s)
 \varrho_{\delta}(x-y)\psi(s,y)\,dx\,dW(r)\Big|^4\,dy\,dv \Big]\notag \\
 & (\text{ By BDG inequality.}) \notag \\
 & \le C \int_v \int_y E\Big[\Big (\int_{0}^T \Big|\int_x  \sigma(x,\tilde{u}(r,x)) \beta(\tilde{u}(r,x)
  -v) \rho_{\delta_0}(r-s) \varrho_{\delta}(x-y)\psi(s,y)\,dx \Big|^2\,dr\Big)^2 \Big]\,dy\,dv\notag \\
 & \le C \int_v \int_{|y|<C_\psi} E\Big[\Big (\int_{0}^T\int_{x}\sigma^2(x,\tilde{u}(r,x)) \beta^2(\tilde{u}(r,x) -v)
  \rho_{\delta_0}^2(r-s) \varrho_{\delta} (x-y)\psi^2 (s,y)\,dx\,dr\Big)^2
 \Big]\,dy\,dv\notag \\
 & \le C \int_v \int_{|y|<C_\psi} E\Big[\int_{0}^T\int_{x}
 \sigma^4(x,\tilde{u}(r,x)) \beta^4(\tilde{u}(r,x)-v) \rho_{\delta_0}^4(r-s) \varrho_{\delta} (x-y)
 \psi^4 (s,y)\,dx \,dr
 \Big]\,dy\,dv\notag \\
 & \le C E\Big[\int_{|y|<C_\psi}\int_{0}^T\int_{x}  \int_{|v|\le C_{\beta}+|\tilde{u}(r,x)|}
 \sigma^4(x,\tilde{u}(r,x))||\beta^\prime||_\infty^4 \rho_{\delta_0}^4(r-s) \varrho_{\delta} (x-y)||\psi||_\infty^4\,dv\,dx \,dr
 \,dy\Big]\notag \\
 & \le C(\beta, \psi) E\Big[\int_{0}^T\int_{x} g^4(x)(1+|\tilde{u}(r,x)|^4)(C_{\beta}+(1+|\tilde{u}(r,x)|))\rho_{\delta_0}^4(r-s)\,dx \,dr\,
 \Big]\notag \\
  & \le C(\beta, \psi) E\Big[\int_{0}^T\int_{x} g^4(x)(1+|\tilde{u}(r,x)|^5) \rho_{\delta_0}^4(r-s)\,dx \,dr\,
 \Big]\notag \\
 & \le C(\beta, \psi) \int_{0}^T(1+E||\tilde{u}(r,\cdot)||^5_5) \rho_{\delta_0}^4(r-s) \,dr
 \Big]\notag \\
  & \le C(\beta,\psi)(1+\sup_{0\le r\le T}E||\tilde{u}(r,\cdot)||_5^5) \int_{0}^T \rho_{\delta_0}^4(r-s) \,dr\notag\\
  & \le C(\beta,\psi)(1+\sup_{0\le r\le T}E||\tilde{u}(r,\cdot)||_5^5) ||\rho_{\delta_0}||_\infty^3\int_0^T \rho_{\delta_0}(r-s) \,dr\notag \\
  & \le \frac{C(\beta,\psi)(1+\sup_{0\le r\le T}E||\tilde{u}(r,\cdot)||_5^5)}{\delta_0^3}\label{l-1}.
 \end{align}
 
 Similarly, we have
\begin{align}
 & E\Big[|| \partial_v J[\beta,\phi_{\delta,\delta_0}](s;\cdot,\cdot)||_4^4\Big] \le
  \frac{C(\beta^{\prime\prime}, \psi )}{\delta_0^3}\label{l-2}  \\
 & E\Big[|| \partial_y J[\beta,\phi_{\delta,\delta_0}](s;\cdot,\cdot)||_4^4\Big] \le
  \frac{C(\beta', \partial_y \psi )}{\delta_0^3} \label{l-3}
\end{align}
Therefore, by \eqref{l-1},\eqref{l-2}, and \eqref{l-3},
\begin{align}
  E\Big[|| J[\beta,\phi_{\delta,\delta_0}](s;\cdot,\cdot)||_{W^{1,4}(\R\times \R)}^4\Big] \le
  \frac{C(\beta', \psi )}{\delta_0^3}.\notag 
 \end{align}
We simply now use Sobolev embedding along with Cauchy-Schwartz inequality and conclude 
 \begin{align}
 \sup_{0\le s\le T} \Big(E\Big[|| J[\beta,\phi_{\delta,\delta_0}](s;\cdot,\cdot)||_{L^\infty(\R\times \R)}^2\Big]\Big) \le
  \frac{C(\beta',  \psi )}{\delta_0^\frac{3}{2}}.
 \end{align}
\end{proof}

Our primary aim is to estimate the expected value of $ J[\beta^\prime,\phi_{\delta,\delta_0}](s;y,v(s,y))$, which we do by estimating the same for $J[\beta^\prime,\phi_{\delta,\delta_0}](s;y,u_\eps(s,y))$ and then passing to the limit. Note that if we directly substitute $v= v(s,y)$ in the formula for  $ J[\beta^\prime,\phi_{\delta,\delta_0}]$, the integrand would no-longer be nonanticipative, and therefore standard methods It\^{o} integrals would no-longer apply. To work around this problem, we proceed as follows.

  Let $\{\rho_l\}_{l > 0} $ be the standard sequence of mollifiers in $\R$ and define 
  \begin{align}
\label{eq:strong-cond-approx} Z_{\eps,\delta,\delta_0,l} &:= \int_{\R}\int_{\Pi_{T}} J[\beta^\prime,\phi_{\delta,\delta_0}](s;y,v)\,\rho_l(u_\eps(s,y)-v)\,dy\,ds\,dv.
\end{align}
We would like to find an upper bound on $E\big[ Z_{\eps,\delta,\delta_0,l} \big] $ as $l,\eps \rightarrow 0$. 
To this end, we claim that for two constants $T_1, T_2\ge 0$ with $T_1<T_2$,

\begin{align}
 E\Big[X_{T_1}\int_{T_1}^{T_2} J(t) dW(t)\Big] = 0\label{eq:conditional_indep}
\end{align}
where $J$ is a predictable integrand and $X(\cdot)$ is an adapted process. The conclusion \eqref{eq:conditional_indep} follows trivially if $J$ is a simple predictable integrand. The general case could be argued by standard approximation technique. 

     If necessary, we extend the process $u_\eps(\cdot,y)$ for negative time simply by $u_\eps(s,y) = u_{\eps} (0,y)$ if $s< 0$. With this convention, it follows from \eqref{eq:conditional_indep} that

\begin{align*}
E\Big[\int_{\R}\int_{\Pi_{T}} J[\beta^\prime,\phi_{\delta,\delta_0}](s;y,v)\,\rho_l(u_\eps(s-\delta_0,y)-v)\,dy\,ds\,dv\Big]=0.
\end{align*}

Hence
\begin{align}
 E[ Z_{\eps,\delta,\delta_0,l}] &= E\Big[\int_{\R}\int_{\Pi_{T}} J[\beta^\prime,\phi_{\delta,\delta_0}](s;y,v)\Big(\rho_l(u_\eps(s,y)-v)
 -\rho_l(u_\eps(s-\delta_0,y)-v)\Big)\,dy\,ds\,dv\Big]. \label{stochastic_estimate_6}
\end{align}
Given $y\in \R$, $u_\eps(\cdot,y)$ satisfies 
\begin{align*}
 du_\eps(s,y)=-\text{div}F_\eps(u_\eps(s,y))ds + \eps  \Delta u_\eps(s,y)\,ds+\sigma_\eps(y,u_\eps(s,y))\,dW(s). 
\end{align*} 

Next,  apply It\^{o}-formula and obtain 
\begin{align*}
 & \rho_l(u_\eps(s,y)-v)-\rho_l(u_\eps(s-\delta_0,y)-v)\\ 
 =& \int_{s-\delta_0}^s \rho_l^{\prime} (u_\eps(\tau,y)-v)\big(-\text{div}F_\eps(u_\eps(\tau,y)) + \eps  \Delta u_\eps(\tau,y)\big) \,d\tau \\
 & + \int_{s-\delta_0}^s  \sigma_\eps(y,u_\eps(\tau,y))\rho_l^{\prime}(u_\eps(\tau,y)-v)\,dW(\tau) +
  \frac{1}{2}\int_{s-\delta_0}^s  |\sigma_\eps(y, u_\eps(\tau,y))|^2 \rho_{l}^{\prime\prime}( u_\eps(\tau,y)-v)\,d\tau \\
 &= -\frac{\partial}{\partial v}\Big[\int_{s-\delta_0}^s \rho_l (u_\eps(\tau,y)-v)\big(-\text{div}F_\eps(u_\eps(\tau,y)) + \eps  \Delta u_\eps(\tau,y)\big) \,d\tau\\
 & + \int_{s-\delta_0}^s  \sigma_\eps(y,u_\eps(\tau,y))\rho_l(u_\eps(\tau,y)-v)\,dW(\tau) +
  \frac{1}{2}\int_{s-\delta_0}^s  \sigma_\eps^2(y, u_\eps(\tau,y)) \rho_{l}^{\prime}( u_\eps(\tau,y)-v)\,d\tau \Big].
\end{align*}

From \eqref{stochastic_estimate_6}, we now have
\begin{align}
  &E[ Z_{\eps,\delta,\delta_0,l}]\notag \\
  =& E\Big[\int_{\R}\int_{\Pi_{T}} J[\beta^\prime,\phi_{\delta,\delta_0}](s;y,v)
  \Big\{-\frac{\partial}{\partial v}\big(\int_{s-\delta_0}^s \rho_l (u_\eps(\tau,y)-v)\big(-\text{div}F_\eps(u_\eps(\tau,y))  \notag \\
  & \hspace{2cm}+ \eps  \Delta u_\eps(\tau,y)\big) \,d\tau 
  + \int_{s-\delta_0}^s  \sigma_\eps(y,u_\eps(\tau,y))\rho_l(u_\eps(\tau,y)-v)\,dW(\tau) \notag \\
 &\hspace{3cm} +\frac{1}{2}\int_{s-\delta_0}^s  \sigma_\eps(y, u_\eps(\tau,y))|^2 \rho_{l}^{\prime}( u_\eps(\tau,y)-v)\,d\tau \big)
 \Big\}\,dy\,ds\,dv\Big] \notag \\
&( \text{By the It\^{o}-product rule and Lemma \ref{lem:differentiation}})\notag \\
& = E\Big[\int_{\R}\int_{\Pi_{T}} J[\beta^{\prime\prime},\phi_{\delta,\delta_0}](s;y,v)\Big(\int_{s-\delta_0}^s \rho_l(u_\eps(\tau,
y)-v)\,\text{div}F_\eps(u_\eps(\tau,y)) \,d\tau\Big)\,dy\,ds\,dv\Big]\notag  \\
& -E\Big[\int_{\R}\int_{\Pi_{T}} J[\beta^{\prime\prime},\phi_{\delta,\delta_0}](s;y,v)\Big(\int_{s-\delta_0}^s \rho_l(u_\eps(\tau,
y)-v)\, \eps  \Delta u_\eps(\tau,y)\big) \,d\tau\Big)\,dy\,ds\,dv\Big]\notag  \\
& -E\Big[\int_{\Pi_{T}}\int_x \int_{s-\delta_0}^s \Big( \int_{\R}\beta^{\prime\prime}(\tilde{u}(r,x)-v)\rho_l(u_\eps(r,y)-v)\,dv\Big)
\sigma(x,\tilde{u}(r,x))\sigma_\eps(y,u_\eps(r,y)) \notag \\
& \hspace{5cm} \times \phi_{\delta,\delta_0}(r,x,s,y)\,dr\,dx\,dy\,ds\Big]\notag \\
& + \frac{1}{2} E\Big[\int_{\R}\int_{\Pi_{T}} J[\beta^{\prime\prime\prime},\phi_{\delta,\delta_0}](s;y,v)\Big\{ \int_{s-\delta_0}^s \sigma_\eps^2(y,u_\eps(\tau,y)
) \rho_{l}( u_\eps(\tau,y)-v)\,d\tau \Big\}dy\,ds\,dv\Big]\notag \\
 &\equiv A_1^{l,\eps}(\delta,\delta_0) + A_2^{l,\eps}(\delta,\delta_0)+  B^{\eps, l}(\delta,\delta_0) + A_3^{l,\eps}(\delta,\delta_0) \label{stochastic_estimate_7}
\end{align}
where
\begin{align}
  A_1^{l,\eps}(\delta,\delta_0) &= E\Big[\int_{\R}\int_{\Pi_{T}} J[\beta^{\prime\prime},\phi_{\delta,\delta_0}](s;y,v)\Big(\int_{s-\delta_0}^s \rho_l(u_\eps(\tau,
y)-v)\,\text{div}F_\eps(u_\eps(\tau,y)) \,d\tau\Big)dy\,ds\,dv\Big]\notag  \\
 A_2^{l,\eps}(\delta,\delta_0) &= -E\Big[\int_{\R}\int_{\Pi_{T}} J[\beta^{\prime\prime},\phi_{\delta,\delta_0}](s;y,v)\Big(\int_{s-\delta_0}^s \rho_l(u_\eps(\tau,
y)-v)\, \eps  \Delta u_\eps(\tau,y)\big) \,d\tau\Big)dy\,ds\,dv\Big]\notag  \\
B ^{l,\eps}(\delta,\delta_0)&=  -E\Big[\int_{\Pi_{T}}\int_x \int_{s-\delta_0}^s \Big( \int_{\R}\beta^{\prime\prime}(\tilde{u}(r,x)-v)\rho_l(u_\eps(r,y)-v)\,dv\Big)
\sigma(x,\tilde{u}(r,x))\notag \\
& \hspace{4.5cm} \times\sigma_\eps(y,u_\eps(r,y))  \phi_{\delta,\delta_0}(r,x,s,y)\,dr\,dx\,dy\,ds\Big]\notag \\
  A_3^{l,\eps}(\delta,\delta_0)&=\frac{1}{2} E\Big[\int_{\R}\int_{\Pi_{T}} J[\beta^{\prime\prime\prime},\phi_{\delta,\delta_0}](s;y,v)\Big\{ \int_{s-\delta_0}^s \sigma_\eps^2(y,u_\eps(\tau,y)
) \rho_{l}( u_\eps(\tau,y)-v)\,d\tau \Big\}\,dy\,ds\,dv\Big]\notag 
\end{align}

Let  $  A_1^{\eps}(\delta,\delta_0):= \lim_{l\goto 0}  A_1^{l,\eps}(\delta,\delta_0)$ and 
$  A_1(\delta,\delta_0) = \limsup_{\eps\downarrow  0} \Big| A_1^{\eps}(\delta,\delta_0) \Big| $.

\begin{lem} \label{finding_A_1} It holds that 
\begin{align}
   A_1(\delta,\delta_0)\goto 0\quad \text{as} ~~ \delta_0 \goto 0.
\end{align}
\end{lem}
 \begin{proof}
 We start by letting
  \begin{align}
  G_\eps(u,v)=\int_0^v \beta^{\prime\prime}(u-r) F_\eps^\prime(r)dr\quad \text{for}~~u,v \in \R.\notag
 \end{align}
It is straightforward  to check that there is a positive integer $p$ such that
\begin{align}
 \sup_{\eps>0}| G_\eps(u,v)| \le C_\beta (1+ |u|^p)\quad \text{for all}~~ u,v \in \R.\label{eq:uniform-growth-estimate}
\end{align}
Let
\begin{align*}
 X_\eps[\phi_{\delta,\delta_0}](s;y,v): & = \int_0^T \int_x \sigma(x,\tilde
 {u}(r,x))  G_\eps(\tilde{u}(r,x),v)\phi_{\delta,\delta_0}(r,x;s,y)\,dx\,dW(r)
\end{align*}
Once again by the same arguments as in Lemma \ref{lem:differentiation}, it holds that
\begin{align}
  &\partial_v X_\eps[\phi_{\delta,\delta_0}](s;y,v) =  \int_0^T \int_x \sigma(x,\tilde
 {u}(r,x))  \partial_v G_\eps(\tilde{u}(r,x),v)\phi_{\delta,\delta_0}(r,x;s,y)\,dx\,dW(r)\notag \\
 &\partial_y X_\eps[\phi_{\delta,\delta_0}](s;y,v)= X_\eps[\partial_y \phi_{\delta,\delta_0}](s;y,v).\notag
\end{align}
Moreover, we can argue as in Lemma \ref{lem:L-infinity estimate} and find a constant $C=C(\beta,\psi)$
  such that
  \begin{align}
   \sup_{\eps > 0}  \sup_{0\le s\le T}\Big(E\Big[|| X_\eps[ \partial_y \phi_{\delta,\delta_0}](s;\cdot,\cdot)||_{L^\infty(\R\times \R)}^2\Big]\Big) \le
  \frac{C(\beta,  \psi )}{\delta_0^\frac{3}{2}}.
  \end{align}
  \noindent{\textbf{Claim:}}
\begin{align}
   A_1^\eps(\delta,\delta_0)=-E\Big[\int_{\Pi_T}\int_{s-\delta_0}^s X_\eps[\partial_y \phi_{\delta,\delta_0}](s;y,u_\eps(\tau,y))
   \,d\tau\,ds\,dy\Big] \label{eq:claim-now}
\end{align}
\noindent{\it Proof of the claim:} We repeatedly use integration by parts and have
\begin{align}
&\int_v \int_{\Pi_T} J[\beta^{\prime\prime},\phi_{\delta,\delta_0}](s,y,v)\Big(\int_{s-\delta_0}^s \rho_l(u_\eps(\tau,y)-v)F_\eps^{\prime}
(v)\partial_y u_\eps(\tau,y)\,d\tau\Big) \,ds\,dy\,dv \notag \\
&=\int_v \int_{\Pi_T} \int_{s-\delta_0}^s\int_0^T \int_x \sigma(x,\tilde
 {u}(r,x))  \beta^{\prime\prime}(\tilde{u}(r,x)-v)F_\eps^\prime(v)\phi_{\delta,\delta_0}(r,x;s,y) \notag \\
 & \hspace{5cm}\times \rho_l(u_\eps(\tau,y)-v)\partial_y u_\eps(\tau,y)\,dW(r)dx
\,d\tau \,ds\,dy\,dv \notag \\
&=\int_v \int_{\Pi_T} \int_{s-\delta_0}^s \partial_v X_\eps[\phi_{\delta,\delta_0}](s;y,v)\rho_l(u_\eps(\tau,y)-v)
\partial_y u_\eps(\tau,y)
\,d\tau \,ds\,dy\,dv \notag \\
&= \int_v \int_{\Pi_T} \int_{s-\delta_0}^s  X_\eps[\phi_{\delta,\delta_0}](s;y,v)\rho_l^\prime (u_\eps(\tau,y)-v)\partial_y u_\eps(\tau,y)
\,d\tau \,ds\,dy\,dv \notag \\
 &= \int_v \int_{\Pi_T} \int_{s-\delta_0}^s  X_\eps[\phi_{\delta,\delta_0}](s;y,v) \partial_y \rho_l (u_\eps(\tau,y)-v)
\,d\tau \,ds\,dy\,dv \notag \\
& =- \int_v \int_{\Pi_T} \int_{s-\delta_0}^s \partial_y  X_\eps[\phi_{\delta,\delta_0}](s;y,v)  \rho_l (u_\eps(\tau,y)-v)
\,d\tau \,ds\,dy\,dv \notag \\
& =- \int_v \int_{\Pi_T} \int_{s-\delta_0}^s   X_\eps[\partial_y \phi_{\delta,\delta_0}](s;y,v)  \rho_l (u_\eps(\tau,y)-v)
\,d\tau \,ds\,dy\,dv. \label{l-4}
\end{align}
We simply let $ l\rightarrow 0$ in both sides of \eqref{l-4} and  obtain
\begin{align}
 &\int_{\Pi_T}\int_{s-\delta_0}^s J[\beta^{\prime\prime},\phi_{\delta,\delta_0}](s;y,u_\eps(\tau,y))\text{div}_y F_\eps(u_\eps(\tau,y))
 \,d\tau\,ds\,dy \notag  \\
 &= - \int_{\Pi_T} \int_{s-\delta_0}^s   X_\eps[\partial_y \phi_{\delta,\delta_0}](s;y,u_\eps(\tau,y))
\,d\tau \,ds\,dy. \label{l-5}
\end{align}
We take expectation in both sides of \eqref{l-5}  and the claim follows.

 Now
 \begin{align}
  \limsup_{\eps\downarrow 0} \Big|A_1^\eps(\delta,\delta_0)\Big| &=  \limsup_{\eps\downarrow 0}  \Big|E\Big[ \int_{\Pi_T} \int_{s-\delta_0}^s   X_\eps[\partial_y \phi_{\delta,\delta_0}](s;y,u_\eps(\tau,y)) 
\,d\tau \,ds\,dy\Big]\Big| \notag \\
& \le C\delta_0 \sup_{0\le s\le T} \sup_{\eps>0}E\Big[ ||  X_\eps[\partial_y \phi_{\delta,\delta_0}](s;\cdot,\cdot) ||_{\infty}\Big] \notag \\
& \le C\delta_0   \sup_{0\le s\le T} \sup_{\eps>0 }\Big(E\Big[ ||  X_\eps[\partial_y \phi_{\delta,\delta_0}](s;\cdot,\cdot) ||_{\infty}^2 \Big] \Big)^\frac{1}{2}\notag \\
& \le C\delta_0  \frac{C(\beta, \phi )}{\delta_0^\frac{3}{4}} \notag \\
& \le C_1(\beta,\phi) \delta_0^\frac{1}{4} \notag.
 \end{align}
 In other words, $ A_1(\delta,\delta_0) \le C_1(\beta,\phi) \delta_0^\frac{1}{4}$, and therefore 
 \begin{align}
   A_1(\delta,\delta_0)\goto 0\quad \text{as} ~~ \delta_0 \goto 0.\notag
\end{align}
\end{proof}
   
  Next, we define 

 \begin{align}
  A_2(\delta,\delta_0):= \limsup_{\eps\downarrow  0} \Big |A_2^{\eps}(\delta,\delta_0)\Big| \quad
  \text{where}~~ A_2^{\eps}(\delta,\delta_0) := \lim_{l\goto 0} A_2^{l,\eps}(\delta,\delta_0) \label{eq:def -A2}
 \end{align}
\begin{lem}\label{finding_A_2} It holds that
\begin{align}
   A_2(\delta,\delta_0)\goto 0\quad \text{as} ~~ \delta_0 \goto 0.
\end{align}
\end{lem}

\begin{proof} 
From the definition of $ A_2^{l,\eps}(\delta,\delta_0) $, it follows that 
\begin{align} 
    A_2^{\eps}(\delta,\delta_0)
    := & \lim_{l\goto \infty} A_2^{l,\eps}(\delta,\delta_0) \notag \\
 =& -E\Big[\int_{\Pi_{T}}\int_{s-\delta_0}^s J[\beta^{\prime\prime},\phi_{\delta,\delta_0}](s;y,u_\eps(\tau,y))
  \eps \Delta u_\eps(\tau,y)
  \,d\tau\,ds\,dy\Big]. \notag 
  \end{align} Hence
  \begin{align}
  &A_2^{\eps}(\delta,\delta_0)\notag\\
  = &E\Big[\int_{\Pi_{T}}\int_{s-\delta_0}^s  \int_0^T \int_x\eps \sigma(x,\tilde{u}(r,x))\Delta_{yy}
  \beta^{\prime}(\tilde{u}(r,x)-u_\eps(\tau,y)) \rho_{\delta_0}(r-s)\,\psi(s,y)\,\varrho_{\delta}(x-y) \notag \\
  & \hspace{6cm} \times \,dx \,dW(r)\,d\tau\,ds\,dy\Big]\notag  \\
  &- E\Big[\int_{\Pi_{T}}\int_{s-\delta_0}^s \int_0^T \int_x  \eps \sigma(x,\tilde{u}(r,x))
  \beta^{\prime\prime\prime}(\tilde{u}(r,x)-u_\eps(\tau,y)) \rho_{\delta_0}(r-s)\,\psi(s,y)\,\varrho_{\delta}(x-y)\notag \\
  &\hspace{6cm}\times |\grad_y u_\eps(\tau,y)|^2\,dx\,dW(r)
  \,d\tau\,ds\,dy\Big]\notag  \\
   &= E\Big[\int_{\Pi_{T}}\int_{s-\delta_0}^s \int_0^T \int_x  \eps \sigma(x,\tilde{u}(r,x))
  \beta^{\prime}(\tilde{u}(r,x)-u_\eps(\tau,y))
  \partial _{yy} \phi_{\delta,\delta_0}(r,x;s,y)\,dx\,dW(r)\,
  \,d\tau\,ds\,dy\Big]\notag  \\
  & - E\Big[\int_{\Pi_{T}}\int_{s-\delta_0}^s J[\beta^{\prime\prime\prime},\phi_{\delta,\delta_0}](s;y,u_\eps(\tau,y))
  \eps  |\grad_y u_\eps(\tau,y)|^2\,d\tau\,ds\,dy \Big]\notag \\
  &\equiv I_1^\eps + I_2^\eps.\notag 
\end{align}
 Now, we use the uniform moment estimates and conclude that
 \begin{align}
 \limsup_{\eps\downarrow  0} |I_1^\eps| =   \lim_{\eps\downarrow  0} |I_1^\eps| = 0.
 \end{align}
 Thus
 \begin{align}
 \limsup_{\eps\downarrow  0}\Big|A_2^{\eps}(\delta,\delta_0)\Big| \le  \limsup_{\eps\downarrow  0} |I_2^\eps|,
 \end{align} and need of the hour is to estimate $I_2^\eps$. Define 
 
 \begin{align*}
  M^t_{s-\delta_0}[\beta^{\prime\prime\prime},\psi,\delta](y,v)=\int_{s-\delta_0}^t \int_x \sigma(x,\tilde{u}(r,x))\beta^{\prime\prime\prime}(\tilde{u}(r,x)-v)
  \psi(s,y)\varrho_\delta(x-y)\,dx\,dW(r),
  \end{align*} where $t \ge s-\delta_0$.  We now invoke It\^{o}-product rule and obtain
  \begin{align*}
   J[\beta^{\prime\prime\prime},\phi_{\delta,\delta_0}](s;y,v)= -\int_{s-\delta_0}^s \rho_{\delta_0}^\prime(t-s)  M^t_{s-\delta_0}[\beta^{\prime\prime\prime},\psi,\delta](y,v)\,dt.
  \end{align*} Therefore
  \begin{align*}
  & || J[\beta^{\prime\prime\prime},\phi_{\delta,\delta_0}](s;\cdot,\cdot)||_{L^{\infty}(\R\times \R)} \le \frac{1}{\delta_0}
   \sup_{s-\delta_0 \le t \le s} ||  M^t_{s-\delta_0}[\beta^{\prime\prime\prime},\psi,\delta](\cdot,\cdot)||_{L^{\infty}(\R\times \R)} .\\
  \end{align*} In other words
  \begin{align}
  & E\Big[\sup_{0\le s \le T}  || J[\beta^{\prime\prime\prime},\phi_{\delta,\delta_0}](s;\cdot,\cdot)||_{L^{\infty}(\R\times \R)}\Big] \notag\\
   \le &\frac{1}{\delta_0} E\Big[\sup_{0\le s\le T;~s-\delta_0\le t<s} ||N_t[\beta^{\prime\prime\prime},\psi,\delta](\cdot,\cdot)
  -N_{s-\delta_0}[\beta^{\prime\prime\prime},\psi,\delta](\cdot,\cdot)||_{L^{\infty}(\R\times \R)}\Big] \label{sup-estimate_2}
   \end{align} where
   \begin{align*}
    N_t [\beta^{\prime\prime\prime},\psi,\delta](y,v)=\int_{0}^t \int_x \sigma(x,\tilde{u}(r,x))\beta^{\prime\prime\prime}(\tilde{u}(r,x)-v)
  \psi(s,y)\varrho_\delta(x-y)\,dx\,dW(r).
   \end{align*}
      By a certain modulus of continuity estimate \cite[Lemma 4.28, P 359]{nualart:2008} for paths of $N_t$, we have  
   \begin{align}
     \label{eq:modulus_of_continuity} E\Big[\sup_{s,t \in [0,T];~ |s-t|< \delta_0} ||N_t [\beta^{\prime\prime\prime},\psi,\delta](\cdot,\cdot)- N_s [\beta^{\prime\prime\prime},\psi,\delta](\cdot,\cdot)||_\infty^p\Big]
    \le C \delta_0^a 
       \end{align} for some $a>0$ and $p>8$.  We combine \eqref{eq:modulus_of_continuity}  and  \eqref{sup-estimate_2} to have
   \begin{align}
    E\Big[ \sup_{0\le s\le T}  || J[\beta^{\prime\prime\prime},\phi_{\delta,\delta_0}](s;\cdot,\cdot)||_{L^{\infty}(\R\times \R)}^p\Big]
    \le C\,\frac{1}{\delta_0^p} \,\delta_0^a \label{sup-estimate_3}
   \end{align} for some $a>0$ and $p>8$.

       Next, we define
   \begin{align*}
    \varLambda_\eps(t)=\int_0^t \eps ||\grad_y u_\eps(r)||_2^2.
   \end{align*} From the moment estimate in Proposition \ref{prop-moment-estimate} we have 
   \begin{align}
   \label{eq:energy_estimate} \sup_{\eps>0} E\Big[|\varLambda_\eps(T)|^p\Big] < \infty,\quad \text{ for}\quad  p=1,2,\cdots, T>0.
   \end{align}

 Finally, we now focus on $I_2^\eps$ and have
   
\begin{align}
  |I_2^\eps| \le &E\Big[\int_0^T\int_{|y|<C_\psi}\int_{s-\delta_0}^s  \sup_{0\le s\le T}||J[\beta^{\prime\prime\prime},\phi_{\delta,\delta_0}](s;\cdot,\cdot)||_\infty
  \eps  |\grad_y u_\eps(\tau,y)|^2\,d\tau\,dy\,ds \Big] \notag \\
  (\text{By Fubini theorem})\quad & \le E\Big[\sup_{0\le s\le T}||J[\beta^{\prime\prime\prime},\phi_{\delta,\delta_0}](s;\cdot,\cdot)||_\infty
   \int_{|y|<C_\psi} \int_{\tau = 0}^T\Big(\int_{s=\tau}^{\tau + \delta_0}\eps 
   |\grad_y u_\eps(\tau,y)|^2 ds\Big)\,d\tau\,dy \Big] \notag \\
   & =\delta_0 E\Big[\sup_{0\le s\le T}||J[\beta^{\prime\prime\prime},\phi_{\delta,\delta_0}](s;\cdot,\cdot)||_\infty
   \int_{|y|<C_\psi} \int_{\tau = 0}^T\eps  |\grad_y u_\eps(\tau,y)|^2 \,d\tau\,dy \Big] \notag \\
    & \le \delta_0 E\Big[\sup_{0\le s\le T}||J[\beta^{\prime\prime\prime},\phi_{\delta,\delta_0}](s;\cdot,\cdot)||_\infty
     \varLambda_\eps(T)\Big] \notag \\
   (\text{By H\"{o}lder with $p>8$})\quad  &\le \delta_0 \Big(E\sup_{0\le s\le T}||J[\beta^{\prime\prime\prime},\phi_{\delta,\delta_0}](s;\cdot,\cdot)||_\infty^p\Big)^\frac{1}{p}
  \Big(E[ | \varLambda_\eps(T)|^q]\Big)^\frac{1}{q}\notag \\
 (\text{By \eqref{sup-estimate_3} and \eqref{eq:energy_estimate}}) \quad & \le C \delta_0^{\tilde{a}},
 \end{align} for some $\tilde{a}>0$. In other words, there exists $\tilde{a}> 0$ such that 
  \begin{align}
 \limsup_{\eps\downarrow  0} |I_2^\eps| \le C(\beta,\psi)\delta_0^{\tilde{a}} \notag
 \end{align}
and hence 
\begin{align}
 A_2{(\delta,\delta_0)} \goto 0 ~ \text{as}~~ \delta_0 \goto 0.
\end{align}
  \end{proof}

 Finally, we  define
\begin{align}
  A_3(\delta,\delta_0)= \limsup_{\eps\downarrow  0}\lim_{l\goto 0} \Big| A_3^{l,\eps}(\delta,\delta_0)\Big|
\end{align}

\begin{lem}\label{finding_A_3} It holds that
\begin{align}
   A_3(\delta,\delta_0)\goto 0\quad \text{as} ~~ \delta_0 \goto 0.\notag
\end{align}
\end{lem}
\begin{proof} By integration by parts, we have
\begin{align*}
  & A_3^{l,\eps}(\delta,\delta_0) \\
  &=\frac{1}{2} E\Big[\int_{\R}\int_{\Pi_{T}} J[\beta^{\prime\prime\prime},\phi_{\delta,\delta_0}](s;y,v)\Big\{ \int_{s-\delta_0}^s \sigma_\eps^2(y,u_\eps(\tau,y)
) \rho_{l}( u_\eps(\tau,y)-v)\,d\tau \Big\}\,dy\,ds\,dv\Big].
  \end{align*} 
  Therefore
  \begin{align}
  | A_3^{l,\eps}(\delta,\delta_0)|  & \le E\Big[  \int_v \int_0^T\int_{{|y|<C_\psi}}\int_{s-\delta_0}^s
   ||J[\beta^{\prime\prime\prime},\phi_{\delta,\delta_0}](s;\cdot,\cdot)||_\infty \sigma_\eps^2(y,u_\eps(\tau,y)) \notag \\
  & \hspace{6cm} \times \rho_l(u_\eps(\tau,y)-v) \,d\tau dy\,ds\,dv \Big]\notag \\
   & = E\Big[ \int_0^T\int_{{|y|<C_\psi}}\int_{s-\delta_0}^s||J[\beta^{\prime\prime\prime},\phi_{\delta,\delta_0}](s;\cdot,\cdot)||_\infty \sigma_\eps^2(y,u_\eps(\tau,y))
  \,d\tau  \,dy\,ds \Big]\notag \\
   & \le E\Big[\int_0^T\int_{{|y|<C_\psi}}\int_{s-\delta_0}^s  ||J[\beta^{\prime\prime\prime},\phi_{\delta,\delta_0}](s;\cdot,\cdot)||_\infty  
  g^2(y)(1+|u_\eps(\tau,y)|^2)\,d\tau \,dy\,ds \Big]\notag \\
   &  \le C\int_0^T\int_{s-\delta_0}^s \Big(E ||J[\beta^{\prime\prime\prime},\phi_{\delta,\delta_0}](s;\cdot,\cdot)||_\infty^2\Big)^\frac{1}{2}   
 \Big( E\int_{|y|<C_\psi} g^4(y)(1+|u_\eps(\tau,y)|^4\,dy \Big)^\frac{1}{2}\,d\tau  \,ds \notag \\
 & \le \frac{ C(\beta,\psi)}{\delta_0^\frac{3}{4}} \int_0^T \int_{s-\delta_0}^s( 1 +E||u_\eps(\tau)||_4^4 )^{\frac 12}\,d\tau
\,ds \notag \\
 & \le C(\beta,\psi) \delta_0^\frac{1}{4} T \Big[ 1 + \sup_{\eps>0}\sup_{0\le t \le T} E ||u_\eps(t,\cdot)||_4^4 \Big]^\frac{1}{2}.
\end{align}
Thus
\begin{align}
 \limsup_{\eps\downarrow  0}\lim_{l\goto 0} \Big| A_3^{l,\eps}(\delta,\delta_0)\Big|\le C(\beta,\psi,T) \delta_0^\frac{1}{4} \notag 
\end{align}
and hence $A_3(\delta,\delta_0) $  has the desired property.
\end{proof}

\begin{lem}\label{finding_B}  It holds that
\begin{align}
 \lim_{\eps\downarrow 0}\lim_{l\downarrow 0} B^{\eps,l}(\delta,\delta_0)= & - E\Big[\int_{\Pi_T}\int_{\Pi_T}  \sigma(x,\tilde{u}(r,x))\sigma(y,v(r,y))\beta^{\prime\prime}(\tilde{u}(r,x)-v(r,y))\notag \\
&\hspace{3cm}\times \phi_{\delta,\delta_0}(r,x,s,y)\,dr\,dx \,dy \,ds \Big]
\end{align}
\end{lem}
\begin{proof}
 Since $|| \beta^{\prime\prime}(\cdot)||_{\infty} < \infty$, we can use dominated convergence theorem and  conclude
 \begin{align}
  &\lim_{l\goto 0} B^{\eps,l}(\delta,\delta_0)=-E\Big[\int_{\Pi_{T}}\int_x \int_{s-\delta_0}^s \beta^{\prime\prime}(\tilde{u}(r,x)-u_\eps(r,y))
\sigma(x,\tilde{u}(r,x))\sigma_\eps(y,u_\eps(r,y)) \notag \\
& \hspace{6cm} \times \phi_{\delta,\delta_0}(r,x,s,y)\,dr\,dx\,dy\,ds\Big].\notag\\
=&-E\Big[\int_{\Pi_{T}}\int_{\Pi_T} \beta^{\prime\prime}(\tilde{u}(r,x)-u_\eps(r,y))
\sigma(x,\tilde{u}(r,x))\sigma_\eps(y,u_\eps(r,y)) \notag \\
& \hspace{6cm} \times \phi_{\delta,\delta_0}(r,x,s,y)\,dr\,dx\,dy\,ds\Big]
 \end{align}
 We use the uniform integrability conditions along with approximation properties of $\sigma_\eps$ and pass to the limit
$ \eps \downarrow 0 $  to obtain
\begin{align}
 \lim_{\eps\downarrow 0}\lim_{l\downarrow 0} B^{\eps,l}(\delta,\delta_0)= & - E\Big[\int_{\Pi_T}\int_{\Pi_T}  \sigma(x,\tilde{u}(r,x))\sigma(y,v(r,y))\beta^{\prime\prime}(\tilde{u}(r,x)-v(r,y))\notag \\
&\hspace{3cm}\times \phi_{\delta,\delta_0}(r,x,s,y)\,dr\,dx \,dy \,ds \Big]\notag.
\end{align}
\end{proof}
 We can now finally wrap up the proof of Lemma \ref{lem:strong-entropy-condition}.
\begin{proof}[Proof of Lemma \ref{lem:strong-entropy-condition}]

We now simply choose $ A(\delta,\delta_0)=A_1(\delta,\delta_0) + A_2(\delta,\delta_0) + A_3(\delta,\delta_0) $.
 Note that, in view of \eqref{eq:strong-cond-approx},
  \begin{align*}
     &E\Big[\int_0^T \int_y\int_0^T\int_x \sigma(x,\tilde{u}(r,x))\beta^\prime(\tilde{u}(r,x)-v)
\phi_{\delta,\delta_0}(r,x,s,y)
 \,dx \,dW(r)\Big|_{v=v(s,y)}\,dy\,ds\Big] \notag\\
  =& \lim_{\eps\downarrow 0}\lim_{l\downarrow 0}  E \big[ Z_{\eps,\delta,\delta_0,l}\big]\\
  =&  \lim_{\eps\downarrow 0}\lim_{l\downarrow 0}  \Big[ A_1^{\eps, l}+A_2^{\eps, l}+A_3^{\eps, l}+B^{\eps, l}(\delta,\delta_0)\Big]\\
  \le &  \limsup_{\eps\downarrow 0}\lim_{l\downarrow 0}|A_1^{\eps, l}(\delta,\delta_0)| + \limsup_{\eps\downarrow 0}\lim_{l\downarrow 0}|A_2^{\eps, l}(\delta,\delta_0)| + \limsup_{\eps\downarrow 0}\lim_{l\downarrow 0}|A_3^{\eps, l}(\delta,\delta_0)|+ \lim_{\eps\downarrow 0}\lim_{l\downarrow 0}B^{\eps, l}\\
   =& A_1(\delta,\delta_0)+ A_2(\delta,\delta_0)+A_3(\delta,\delta_0)+\lim_{\eps\downarrow 0}\lim_{l\downarrow 0}B^{\eps, l}\\
   = & A(\delta,\delta_0) - E\Big[\int_{\Pi_T}\int_{\Pi_T}  \sigma(x,\tilde{u}(r,x))\sigma(y,v(r,y))\beta^{\prime\prime}(\tilde{u}(r,x)-v(r,y))\notag \\
&\hspace{3cm}\times \phi_{\delta,\delta_0}(r,x,s,y)\,dr\,dx \,dy \,ds \Big]
  \end{align*}
  where we have used Lemma \ref{finding_B}.  Furthermore, by  Lemmas 
  \ref{finding_A_1}-\ref{finding_A_3}, the function $A(\delta,\delta_0)$ has the desired property as $\delta_0 \rightarrow 0$.    
\end{proof}

We have seen from Lemma \ref{lem:entropysolution} that $v(t,x)=\bar{u}(t,x)$ is a stochastic entropy solution. Moreover,
we conclude  from Lemma \ref{lem:strong-entropy-condition} that $\bar{u}(t,x)$ is indeed a stochastic strong entropy solution of \eqref{eq:brown_stochconservation_laws}-\eqref{initial_cond}, which completes the proof of Theorem \ref{thm:existence}.

\section{A critique on the strong-in-time formulation }

In this final section, we will contest the suitability of strong-in-time formulation of \cite{nualart:2008} and try  to make a case for weak-in-time formulation. However, the issues that we are going to raise are purely technical in nature and do not any way disturb the broader message of \cite{nualart:2008}. We could not have emphasized more on the fact the article \cite{nualart:2008} is no less than a milestone in the area.

     For any $L^p$-valued solution process $u(\cdot, x)$ with continuous sample paths, it is easy to see that the strong-in-time and weak-in-time formulations are equivalent to each other. Furthermore, if it is not established that the solution process has continuous paths then weak-in-time formulation is certainly a more appropriate way to move forward. Just as in the deterministic case, the authors use vanishing viscosity method for existence in \cite{nualart:2008} and attempts have been made in  \cite{nualart:2008} to justify that the vanishing viscosity limit has continuous sample paths when treated as a $\mathcal{M}_0$-valued process. To be more precise, it is shown in \cite[Lemma 4.23, P 355]{nualart:2008}  that 
     \begin{align}
     \label{eq:conv-feng-nualart} \lim_{t\downarrow s}E\big[r(\mu_0(t), \mu_0(s))\big] = 0,
     \end{align} and a claim has been made that \eqref{eq:conv-feng-nualart} implies that $\mu_0(\cdot)$ has continuous sample paths as $\mathcal{M}_0$ valued process. We strongly disagree with the derivation of \eqref{eq:conv-feng-nualart} in the proof of  \cite[Lemma 4.23, P 355]{nualart:2008}. Moreover,  the claim that $\mu_0(\cdot)$ has continuous sample paths because of \eqref{eq:conv-feng-nualart} is also wrong. In fact, we make a counter claim that an estimate of type \eqref{eq:conv-feng-nualart} may not imply path continuity. To see this, let $N_t$ be the usual Poisson process with parameter $\lambda > 0$. Then 
        \begin{align}
\lim_{t \rightarrow s}E\big[d_{\R} (N_t, N_s)\big] = \lim_{t \rightarrow s}E\big[|N_t- N_s|\big] =   \lim_{t \rightarrow s} \lambda |t-s| =0,
     \end{align} but $N_t$ clearly  does not have continuous sample paths. Therefore, $\mu_0(\cdot)$ cannot be claimed to have continuous sample paths on the basis of \eqref{eq:conv-feng-nualart} alone. This invalidates the claim in \cite[Lemma 4.22, P 355]{nualart:2008} that $\mu_0(t)$ has trajectories in $C\big([0,\infty), \mathcal{M}_0\big)$, and puts a question mark next the entropy inequality  \cite[(74), P 355]{nualart:2008}. 
     
      Moreover, the proof \eqref{eq:conv-feng-nualart} in  \cite[Lemma 4.23, P 355]{nualart:2008} is incorrect due to the lapses in  \cite[Lemma 4.15, P 343]{nualart:2008}. 
     To elaborate on this point,  let us look at the proof  \cite[Lemma 4.15, P 343]{nualart:2008} where it is shown that 
          \begin{align}
       \label{eq:conv-feng-nualart-1}      \lim_{\eps\downarrow 0} E[d(\mu_\eps(\cdot), \mu_0(\cdot))]= \lim_{\eps\downarrow 0}\int_0^\infty e^{-t}E\big[\min\big(1, r(\mu_\eps(t), \mu_0(t))\big)\big]\,dt=0.
             \end{align}
             
  Clearly, \eqref{eq:conv-feng-nualart-1} only  implies that $\lim_{\eps\downarrow 0} E[r(\mu_\eps(t), \mu_0(t))]=0$ {\bf for almost every} $t\ge 0$, contrary to the claim in  \cite[Lemma 4.15, P 343]{nualart:2008} that $\lim_{\eps\downarrow 0} E[r(\mu_\eps(t), \mu_0(t))]=0$ {\bf for every} $t\ge 0$.  This jeopardizes  the claim that 
  \[\lim_{\eps\downarrow 0} \Big(\mu_\eps(t_1),........, \mu_\eps(t_m)\Big) = \Big(\mu_0(t_1),........, \mu_0(t_m) \Big)~\text{in probability}\] for each $0\le t_1\le \cdots \le t_m$. We object to the wording `{\it for each $0\le t_1\le \cdots \le t_m$}'. In our view, the correct wording should be  `{\it  $0\le t_1\le \cdots \le t_m$ where $t_i$'s are chosen from a set of full Lebesgue measure in $[0,\infty)$ }'. Hence, one would only be allowed to pass to the limit in $\eps$ in   \cite[(73), P 354]{nualart:2008} for {\bf almost every} $(t,s)\in [0,\infty)\times [0,\infty)$, and   \cite[(74), P 355]{nualart:2008} would be valid only for  almost every  $(t,s)\in [0,\infty)\times [0,\infty)$. 
               
             Therefore, it is fair to say that the vanishing viscosity limit does not have sufficiently clear point-wise picture in time for its paths, and it is worthwhile to go for the weak-in-time entropy formulation for  \eqref{eq:brown_stochconservation_laws}. It is worth mentioning that it may well be possible to prove the path continuity for the entropy solution, but the methods of \cite{nualart:2008} are not adequate for that. Also, the weak-in-time formulation has an immediate correspondence with kinetic formulation of \cite{Vovelle2010}. Though the model  that is considered in \cite{Vovelle2010} deals with periodic solutions, kinetic solutions are claimed to have continuous paths. It may be possible to develop kinetic solution framework in a general case, which might help to establish path continuity for our framework.





\begin{thebibliography}{99}

 \bibitem{Balder}
E.~J.~Balder.
\newblock Lectures on Young measure theory and its applications in economics.
\newblock{\em Rend. Istit. Mat.Univ. Trieste}, 31 Suppl. 1:1-69, 2000.

\bibitem{vallet2012}
C. Bauzet, G. Vallet and P. Wittbold.
\newblock The Cauchy problem for  conservation law with a multiplicative stochastic perturbation.
\newblock{\em J. Hyperbolic Differ. Equ. 9 (2012), no.4}, 661-709.

 \bibitem{Chen-karlsen2011}
G.~Q.~ Chen, Q.~ Ding, and Kenneth H. Karlsen.
\newblock On nonlinear stochastic balance laws.
\newblock{\em  Arch. Ration. Mech. Anal.} 204 (2012), no. 3, 707Ð743.

\bibitem{dafermos}
C.~M. ~Dafermos.
\newblock {\em Hyperbolic conservation laws in continuum physics}, volume 325
  of {\em Grundlehren der Mathematischen Wissenschaften [Fundamental Principles
  of Mathematical Sciences]}.
\newblock Springer-Verlag, Berlin, 2000.



\bibitem{Vovelle2010}
A.~ Debussche and J. Vovelle. 
\newblock Scalar conservation laws with stochastic forcing. 
\newblock{\em J. Funct. Analysis}, 259 (2010), 1014-1042. 

\bibitem{xu}
Z.~Dong and T.~G. Xu.
\newblock One-dimensional stochastic {B}urgers equation driven by {L}{\'e}vy
  processes.
\newblock {\em J. Funct. Anal.}, 243(2):631--678, 2007.

\bibitem{evansweak}
Lawrence~C. Evans.
\newblock {\em Weak convergence methods for nonlinear partial differential
  equations}, volume~74 of {\em CBMS Regional Conference Series in
  Mathematics}.
\newblock Published for the Conference Board of the Mathematical Sciences,
  Washington, DC, 1990.
  
\bibitem{evanspde}
Lawrence~C. Evans.
\newblock {\em Partial differential equations}, volume~19 of {\em Graduate
  Studies in Mathematics}.
\newblock American Mathematical Society, Providence, RI, second edition, 2010.

\bibitem{nualart:2008}
Jin Feng and David Nualart.
\newblock Stochastic scalar conservation laws.
\newblock {\em J. Funct. Anal.}, 255(2):313--373, 2008.


\bibitem{godu}
Edwige Godlewski and Pierre-Arnaud Raviart.
\newblock {\em Hyperbolic systems of conservation laws}, volume 3/4 of {\em
  Math{\'e}matiques \& Applications (Paris) [Mathematics and Applications]}.
\newblock Ellipses, Paris, 1991.


\bibitem{risebroholden1997}
H.~ Holden and N.~H.~ Risebro.
\newblock Conservation laws with random source.
\newblock{\em Appl. Math. Optim}, 36(1997), 229-241.

\bibitem{horo1}
J. ~Horowitz.
\newblock Measure-valued random processes.
\newblock{\em Probability Theory and Related Fields}, 70(1985), pp 213-236.

\bibitem{horo2}
J. ~Horowitz.
\newblock Gaussian random measures
\newblock{\em Stochastic Process. Appl. }, 22(1986), pp 129-133.


\bibitem{kallen1}
O. Kallenberg
\newblock $L_p$-intensities of random measures.
\newblock{ \em Stochastic Process. Appl.} 9 (1979), no. 2, 155Ã161.

\bibitem{kallenberg}
Olav Kallenberg.
\newblock {\em Foundations of modern probability}.
\newblock Probability and its Applications (New York). Springer-Verlag, New
  York, second edition, 2002.




\bibitem{KIm2005}
J.~U.~ Kim.
\newblock On a stochastic scalar conservation law, 
\newblock{Indiana Univ. Math. J.}  52 (1) (2003) 227Ã256. 







\bibitem{malek}
J.~M{\'a}lek, J.~Ne\v{c}as, M.~Rokyta, and M.~R{{\o}circ{u}}\v{z}i\v{c}ka.
\newblock {\em Weak and measure-valued solutions to evolutionary {PDE}s},
  volume~13 of {\em Applied Mathematics and Mathematical Computation}.
\newblock Chapman \& Hall, London, 1996.

\bibitem{metivier}
Michel M{\'e}tivier.
\newblock {\em Semimartingales}, volume~2 of {\em de Gruyter Studies in
  Mathematics}.
\newblock Walter de Gruyter \& Co., Berlin, 1982.
\newblock A course on stochastic processes.

\bibitem{parthasarathy}
K.~R. Parthasarathy.
\newblock {\em Probability measures on metric spaces}.
\newblock Probability and Mathematical Statistics, No. 3. Academic Press Inc.,
  New York, 1967.



\bibitem{Protter1990}
Philip Protter.
\newblock Stochastic Integration and Differential Equations.
\newblock Springer-Verlag, Berlin, 1990.



\bibitem{Sinai1997}
Khanin, Sinai.
\newblock Invariant measures for Burgers equation
with random forcing
\newblock{\em Annals of Math (2).} 151 (2000), no. 3, 877Ð960.

\bibitem{stroock}
Daniel~W. Stroock and S.~R.~Srinivasa Varadhan.
\newblock {\em Multidimensional diffusion processes}, volume 233 of {\em
  Grundlehren der Mathematischen Wissenschaften [Fundamental Principles of
  Mathematical Sciences]}.
\newblock Springer-Verlag, Berlin, 1979.

\bibitem{vallet2010}
G. Vallet  and P. Wittbold.
\newblock On a stochastic first-order hyperbolic equation in a bounded domain. 
\newblock {\em Infin. Dimens. Anal. Quantum Probab. Relat. Top.} 12 (2009), no. 4, 613Ã651. 



\bibitem{walsh}
John~B. Walsh.
\newblock An introduction to stochastic partial differential equations.
\newblock In {\em {\'E}cole d'{\'e}t{\'e} de probabilit{\'e}s de
  {S}aint-{F}lour, {XIV}---1984}, volume 1180 of {\em Lecture Notes in Math.},
  pages 265--439. Springer, Berlin, 1986.


\end{thebibliography}
\end{document}